\documentclass[a4paper,reqno,12pt]{amsart}
\usepackage[left=2.8cm,right=2.8cm,top=2.86cm,bottom=2.86cm]{geometry}
\usepackage{amsmath,amssymb,latexsym,esint,cite,verbatim,wasysym,mathrsfs}
\usepackage{microtype}
\usepackage{color,enumitem,graphicx}
\usepackage[colorlinks=true,urlcolor=blue, citecolor=red,linkcolor=blue,
linktocpage,pdfpagelabels,bookmarksnumbered,bookmarksopen]{hyperref}
\usepackage[hyperpageref]{backref}
\usepackage[english]{babel}
\usepackage{tikz}

\usepackage{amsfonts}
\usepackage{amsthm}

\newcommand*\diff{\mathop{}\!\mathrm{d}}

\begin{document}

\title[On a critical Kirchhoff $p(\cdot)$-biharmonic problem]
{On a $p(\cdot)$-biharmonic problem of Kirchhoff type involving critical growth}

\author[N.T.\ Chung and K.\ Ho]{Nguyen Thanh Chung and Ky Ho}

\address[N.T.\ Chung]{
	Department of Mathematics, \newline
	Quang Binh University, 312 Ly Thuong Kiet, Dong Hoi, Quang Binh, Vietnam} 
\email{ntchung82@yahoo.com}

\address[K.\ Ho]{Institute of Applied Mathematics,\newline University of Economics Ho Chi Minh City, 59C, Nguyen Dinh Chieu Street, District 3, Ho Chi Minh City, Vietnam}
\email{kyhn@ueh.edu.vn}

\begin{abstract}

We establish a concentration-compactness principle for the Sobolev space $W^{2,p(\cdot)}(\Omega)\cap W_0^{1,p(\cdot)}(\Omega)$ that is a tool for overcoming the lack of compactness of the critical Sobolev imbedding. Using this result we obtain several existence and multiplicity results for a class of Kirchhoff type problems involving $p(\cdot)$-biharmonic operator and critical growth. 

\end{abstract}
\maketitle \numberwithin{equation}{section}
\newtheorem{theorem}{Theorem}[section]
\newtheorem{lemma}[theorem]{Lemma}
\newtheorem{definition}[theorem]{Definition}
\newtheorem{claim}[theorem]{Claim}
\newtheorem{proposition}[theorem]{Proposition}
\newtheorem{remark}[theorem]{Remark}
\newtheorem{corollary}[theorem]{Corollary}
\newtheorem{example}[theorem]{Example}
\allowdisplaybreaks


\newcommand{\abs}[1]{\left\lvert#1\right\rvert}
\newcommand{\norm}[1]{|\!|#1|\!|}
\newcommand{\Norm}[1]{\mathinner{\Big|\!\Big|#1\Big|\!\Big|}}
\newcommand{\curly}[1]{\left\{#1\right\}}
\newcommand{\Curly}[1]{\mathinner{\mathopen\{#1\mathclose\}}}
\newcommand{\round}[1]{\left(#1\right)}
\newcommand{\bracket}[1]{\left[#1\right]}
\newcommand{\scal}[1]{\left\langle#1\right\rangle}
\newcommand{\Div}{\text{\upshape div}}
\newcommand{\dotsr}{\dotsm}
\newcommand{\R}{{\mathbb R}}
\newcommand{\Zero}{{\mathbf0}}
\newcommand{\ra}{\rightarrow}
\newcommand{\ran}{\rangle}
\newcommand{\lan}{\langle}
\newcommand{\ol}{\overline}
\newcommand{\N}{{\mathbb N}}
\newcommand{\e}{{\varepsilon}}
\newcommand{\al}{{\alpha}}
\newcommand{\la}{{\lambda}}
\newcommand{\Ne}{\mathcal{N}}

\section{Introduction}
In this paper, we study the existence of solutions to the following problem involving critical growth
\begin{equation}\label{e1.1}
\begin{cases}
\Delta^2_{p(x)}u - M\left(\int_\Omega \frac{1}{p(x)}|\nabla u|^{p(x)}\,\diff x\right)\Delta_{p(x)}u = \lambda f(x,u) +|u|^{q(x)-2}u ~\text{in}~ \Omega ,\\
  u=\Delta u=0 ~\text{on} ~ \partial \Omega ,
  \end{cases}
\end{equation}
where $\Omega $ is a bounded domain of $\mathbb{R}^{N}$ $(N\geq 3)$ with a Lipschitz boundary $\partial \Omega$; $\Delta_{p(x)}^2u := \Delta (|\Delta u|^{p(x)-2} \Delta u)$ is the operator of fourth order called the $p(\cdot)$-biharmonic operator, $\Delta_{p(x)}u := \operatorname{div} \left(|\nabla u|^{p(x)-2}\nabla u\right)$ is the $p(\cdot)$-Laplacian; 
$p,q \in C_+(\overline\Omega):=\big \{h\in C(\overline\Omega):
1<h^-:=\min_{x\in\overline\Omega}h(x)\leq h^+:=\max_{x\in\overline\Omega}h(x)<\infty\big\}$ such that 
$p^+<N/2$ and $p(x)<q(x)\leq p_2^\ast(x):=\frac{Np(x)}{N-2p(x)}$ for all $ x\in \overline{\Omega}$; 
$M: \, \R^+_0:=[0, \infty) \to \mathbb{R}^+$ is a continuous function; $f:\Omega\times \mathbb{R} \to \mathbb{R}$ is a Carath\'eodory function; and $\lambda$ is a positive real parameter. 

We know that problem \eqref{e1.1} is strongly related to extensible beam type equations and stationary Berger plate equations, which are known as the generalizations of Kirchhoff type equations. Indeed, Woinowsky-Krieger \cite{WK1950} studied the following fourth equation in one dimension
\begin{equation}\label{e.beam}
\frac{\partial^2 u}{\partial t^2} + \frac{EI}{\rho}\frac{\partial^4 u}{\partial x^4} - \left(\frac{H}{\rho} + \frac{EA}{2\rho L}\int_0^L \left|\frac{\partial u}{\partial x}\right|^2\,dx\right)\frac{\partial^2 u}{\partial x^2} = 0,
\end{equation}
which was proposed to modify the theory of the dynamic Euler-Bernoulli beam. There, $L$ is the length of the beam in the rest position, $E$ is the Young modulus of the material, $I$ is the cross-sectional moment of inertia, $\rho$ is the mass density, $H$ is the tension in the rest position and $A$ is the cross-sectional area. In the special case $I=0$, \eqref{e.beam} is called a Kirchhoff type equation, see \cite{K1883}. After that, Berger \cite{B1955} studied the following von Karman plate equation in two dimension
\begin{equation}\label{e.plate}
\frac{\partial^2 u}{\partial t^2} + \Delta^2 u + \left(Q + \int_\Omega |\nabla u|^2\,dx\right)\Delta u = f(u,u_t,x)
\end{equation}
that is known as the Berger plate model. Here, $\Delta^2(\cdot)$ and $\Delta (\cdot)$ are usually called the biharmonic operator and the Laplace operator, respectively. Equation \eqref{e.plate} describes large deflection of plate, where the parameter $Q$ describes in-plane forces applied to the plate and the function $f$ represents transverse loads, which may depend on the displacement $u$ and the velocity $u_t$. Because of the importance of equations \eqref{e.beam} and \eqref{e.plate} in engineering, physics and material mechanics, many studies on the existence of solutions as well as their properties have been made in the last decades. From the mathematical point of view, problem \eqref{e1.1} can be seen as a generalization of the stationary problem associated with \eqref{e.beam} in one dimension and equation \eqref{e.plate} in two dimension.

Now, we describes some recent results that lead us to study problem \eqref{e1.1}. We point out that problem \eqref{e1.1} in the case of constant exponent $p(\cdot)\equiv p$, namely,
\begin{equation}\label{e1.4}
\begin{cases}
\Delta^2_pu - M\left(\int_\Omega |\nabla u|^p\,\diff x\right)\Delta_pu = g(x,u) ~\text{in}~ \Omega ,\\
  u=\Delta u=0 ~\text{on} ~ \partial \Omega ,
  \end{cases}
\end{equation}
has been studied by many mathematicians. We start with the paper by Ma \cite{M2007}, which can be seen as one of interesting papers on the existence of solutions for fourth-order problems of Kirchhoff type. There, the author studied the existence and multiplicity of positive solutions for the following problem 
\begin{equation}\label{e1.5}
\begin{cases}
u^{''''} - M\left(\int_0^1 |u'|^2\,\diff x\right)u'' = q(x) h(x,u,u'),\\
u(0) = u(1) = u^{''}(0) = u^{''}(1) = 0,
  \end{cases}
\end{equation}
using the Krasnosel'skii fixed point theorem in cones of ordered Banach spaces, $h \in C([0,1]\times [0,\infty)\times \R)$, $M\in C([0,\infty))$ are nonnegative real functions and $q\in C[0,1]$ is a positive real function. In \cite{WA2012}, Wang et al. first considered problem \eqref{e1.4} with $p=2$ using variational methods. In order to show the existence of solutions, the Kirchhoff function $M$ is assumed to be positive and bounded on $[0,\infty)$ while the nonlinearity $g$ satisfies the Ambrosetti-Rabinowitz type condition. Some restrictions on the functions $M$ and $g$ are eliminated in the later paper \cite{WAA2014} in which the Kirchhoff function is of the form $M(t) = a+bt$, $a >0$, $b\geq 0$ and the mountain pass techniques as well as truncation methods are applied. Later on, there have been many papers concerning problem \eqref{e1.4} where $\Omega$ is a bounded domain or the whole $\R^N$. We refer the reader to the papers \cite{BKLP2019, FKH2014, LZ2016, WL2020, XC2015} for the existence of the solutions and \cite{SW2019, ZTCZ2016} for the properties of solutions. However, there are very few papers considering problem \eqref{e1.4} in the critical case, that is, the nonlinearity $g$ has critical growth 
(see, for example,  \cite{CF2014, Fig-Nas.EJDE2016, LZ2016}).


Recently, the study of partial differential equations with variable exponents has received a lot of attention. The reason of such interest starts from the study of the role played by their applications in mathematical modelling of non-Newtonian fluids, in particular, the electro-rheological fluids and of other phenomena related to image processing, elasticity and the flow in porous media, see \cite{Chen,Ruzicka}. For more information, the reader can consult the important monographs due to Diening et al. \cite{Diening} and R\u{a}dulescu et al. \cite{Radulescu}. Many authors have been also interested in problems with variable exponents involving the $p(\cdot)$-Laplacian, Kirchhoff terms as well as a critical growth (see, for example, \cite{Alves1, Bonder, CC2015, F2009, Ho-Sim, HMR2017}). In the case of problems involving the $p(\cdot)$-biharmonic operator, we refer to \cite{AA2009, BR2016, BRR2016, Chung, K2015}. In the above papers, we are particularly interested in the works where the critical problems are discussed. 

One of the main features of elliptic equations involving critical growth is the lack of compactness arising in connection with the variational approach. In order to overcome the lack of compactness, Lions \cite{Lions} introduced the method using the so-called concentration-compactness principle to show a Palais-Smale ((PS), for short) sequence is precompact. The variable exponent version of the Lions concentration-compactness principle for a bounded domain was independently obtained in \cite{F2009} and \cite{Bonder}. Since then, many authors have applied these results to study critical elliptic problems involving variable exponents. As far as we know, while many authors are interested in the study of problems involving $p(\cdot)$-biharmonic operator in both local and nonlocal cases, there is no paper mentioning the critical case of nonlinearity. So, the goal of our paper is to fill this gap.

Our study here is particularly inspired by the contributions in recent works \cite{Bonder, Chung, CC2015, F2013, Fig-Nas.EJDE2016, Ho-Sim, HMR2017}. In \cite{F2013, Fig-Nas.EJDE2016}, Figueiredo et al. studied a class of Kirchhoff type problems with critical growth involving the Laplace operator $\Delta(\cdot)$ or the biharmonic operator $\Delta^2(\cdot)$. Meanwhile, based on the concentration-compactness principle by Bonder and Silva \cite{Bonder}, Correa et al. \cite{CC2015} and Hurtado et al. \cite{HMR2017} developed the above results (for the Laplace operator $\Delta(\cdot)$) in the variable exponent framework. In this paper, we will study the existence of solutions for $p(\cdot)$-biharmonic problem \eqref{e1.1} with critical growth. As we will see in the next sections, there are three main difficulties in our situation. The first one comes from the appearance of the nonlocal term $M\left(\int_\Omega \frac{1}{p(x)}|\nabla u|^{p(x)}\,\diff x\right)$, which causes some mathematical difficulties because \eqref{e1.1} is no longer a pointwise identity. The second one is that our problem involves the critical growth that makes it lack of compactness since the imbedding $W^{2,p(\cdot)}(\Omega)\cap W_0^{1,p(\cdot)}(\Omega) \hookrightarrow L^{q(\cdot)}(\Omega)$ is not compact. Finally, we can see that problem \eqref{e1.1} is considered with non-standard growth conditions. This leads to the fact that the operators appeared in the problem are not homogeneous. To overcome the above difficulties, we first establish a concentration-compactness principle for the Sobolev space $W^{2,p(\cdot)}(\Omega)\cap W_0^{1,p(\cdot)}(\Omega)$ which is one of the main results of our paper, see Theorem \ref{Theo.ccp}. Then, applying variational methods combined with the genus theory, we obtain some existence and multiplicity results for the problems with a generalized concave-convex or $p(\cdot)$-superlinear term, see Theorems \ref{Theo.cc} and \ref{Theo.sl}. In particular, we introduce a new class of nonlinearities which is a generalization of concave-convex type of nonlinearities. For this reason, we believe Theorem \ref{Theo.cc} is new even in the constant exponent case.

The paper is organized as follows. In Section~\ref{MainResults}, we state our main results: the concentration-compactness principle for the Sobolev space $W^{2,p(\cdot)}(\Omega)\cap W_0^{1,p(\cdot)}(\Omega)$ and the existence results for problems with a generalized concave-convex or $p(\cdot)$-superlinear term. In Section~\ref{Pre}, we briefly review Lebesgue and Sobolev spaces with variable exponent, which is the framework of our mathematical analysis. Section~\ref{CCP} is devoted to the proof of the aforementioned concentration-compactness principle. In the last section, Section~\ref{Existence}, we prove existence results using our concentration-compactness principle, a truncation technique and genus theory. 

\section{Main Results}\label{MainResults}
In order to state our main results, we first present a framework of our analysis.  For $p\in C_+(\overline\Omega)$ given, define the variable
exponent Lebesgue space $L^{p(\cdot)}(\Omega)$ as
$$
L^{p(\cdot)}(\Omega) := \left \{ u : \Omega\to\mathbb{R}\  \hbox{is measurable},\ \int_\Omega |u(x)|^{p(x)} \;\diff x < \infty \right \},
$$
endowed with the Luxemburg norm
$$
\abs{u}_{p(\cdot)}:=\inf\left\{\lambda >0:
\int_\Omega
\Big|\frac{u(x)}{\lambda}\Big|^{p(x)}\;\diff x\le1\right\}.
$$
Let $k\in\N$. The variable exponent Sobolev space $W^{k,p(\cdot)}(\Omega) $ is defined as
\[
W^{k,p(\cdot)}(\Omega) :=\{u\in L^{p(\cdot) }(\Omega) :
|D^\alpha u|\in L^{p(\cdot) }(\Omega ),\ |\alpha|\leq k \},
\]
endowed with the norm
\[
\|u\|_{W^{k,p(\cdot)}(\Omega )}:=\sum_{|\alpha|\leq k}\big||D^\alpha u|\big|_{p(\cdot)}.
\]
Define $W_0^{k,p(\cdot)}(\Omega)$ as the closure of $C_c^\infty(\Omega)$ in $W^{k,p(\cdot)}(\Omega)$. It is well known that $W^{k,p(\cdot)}(\Omega )$ and $W_0^{k,p(\cdot)}(\Omega)$ are reflexive separable Banach spaces (see, for example, \cite{Diening,Radulescu}). Note that on $W^{2,p(\cdot)}(\Omega)\cap W_0^{1,p(\cdot)}(\Omega),$ the norms $\|u\|_{W^{2,p(\cdot)}(\Omega )}$ and $\|u\|:=|\Delta u|_{p(\cdot)}$ are equivalent due to \cite[Theorem 4.4]{Zang-Fu}. We will look for solutions to problem~\eqref{e1.1} in space: 
$$X:=\left(W^{2,p(\cdot)}(\Omega)\cap W_0^{1,p(\cdot)}(\Omega),\|\cdot\|\right).$$

We denote by $C^{\text{log}}_+(\overline{\Omega})$ the set of all functions $h$ in $C_+(\overline{\Omega})$ that are log-H\"older continuous, namely,
$$\underset{\underset{0<|x-y|<\frac{1}{2}}{x,y\in\overline{\Omega}}}{\sup}|h(x)-h(y)|\log\frac{1}{|x-y|}<\infty.$$
\noindent Throughout this paper, we assume that the exponents $p$ and $q$ satisfy
\begin{itemize}
	\item[$(\mathcal{C})$] $p\in C^{\text{log}}_+(\overline{\Omega})$, $p^+<N/2$, $q\in C_+(\overline{\Omega})$, $p(x)<q(x)\leq p_2^\ast(x)$ for all $x\in\overline{\Omega},$ and $$ \mathcal{A}:=\left\{x\in \overline{\Omega}: \, q(x)=p_2^\ast(x)\right\}\ne \emptyset.$$
\end{itemize}
Under this assumption, it holds that $X\hookrightarrow L^{q(\cdot)}(\Omega)$ and hence, 
\begin{equation}\label{S}
S:=\underset{\phi\in X\setminus\{0\}}{\inf}
\frac{\|\phi\|}{| \phi |_{q(\cdot)}}>0
\end{equation}
(see Proposition~\ref{critical.imb} below). In order to state our first main result, let $\mathcal{M}(\overline{\Omega})$ denote the space of Radon measures on $\overline{\Omega}$; that is, the dual space of $C(\overline{\Omega})$. By Riesz representation theorem, for each $\mu\in \mathcal{M}(\overline{\Omega}),$ there is a unique signed Borel measure on $\overline{\Omega}$, still denoted by $\mu$, such that
$$\langle \mu,f\rangle=\int_{\overline{\Omega}}f\diff\mu,\quad \forall f\in C(\overline{\Omega}).$$
We identify $L^1(\Omega)$ with a subspace of $\mathcal{M}(\overline{\Omega})$ through the imbedding
$T:\ L^1(\Omega)\to \mathcal{M}(\overline{\Omega})$  defined by 
$$\langle Tu,f\rangle=\int_{\Omega}uf\diff x \ \ \forall u\in L^1(\Omega), \ \forall f\in C(\overline{\Omega})$$
(see, for example., \cite[p. 116]{Brezis-book}). In what follows, the notations $u_n\to u$ (resp. $u_n \rightharpoonup u,u_n \overset{\ast }{\rightharpoonup }u$) stand for the term $u_n$ strongly (resp. weakly, weakly-$\ast$) converges to $u$ as $n \to \infty$ in an appropriate space. 

The following theorem is a concentration-compactness principle for high order Sobolev spaces with variable exponent, which is essential for our arguments in dealing with problem~\eqref{e1.1} via variational methods.  
\begin{theorem}\label{Theo.ccp} {\rm \textbf{(The concentration-compactness principle for $X=W^{2,p(\cdot)}(\Omega)\cap W_0^{1,p(\cdot)}(\Omega)$)}}
	Assume that $(\mathcal{C})$ holds. Let $\{u_n\}_{n\in\mathbb{N}}$ be a bounded sequence in
	$X$ such that
	\begin{eqnarray*}
		u_n &\rightharpoonup& u \quad \text{in}\quad  X,\\
		|\Delta u_n|^{p(\cdot)} &\overset{\ast }{\rightharpoonup }&\mu\quad \text{in}\quad \mathcal{M}(\overline{\Omega}),\\
		|u_n|^{q(\cdot)}&\overset{\ast }{\rightharpoonup }&\nu\quad \text{in}\quad \mathcal{M}(\overline{\Omega}).
	\end{eqnarray*}
	Then, there exist $\{x_i\}_{i\in I}\subset \mathcal{A}$ of distinct points and $\{\nu_i\}_{i\in I}, \{\mu_i\}_{i\in I}\subset (0,\infty),$ where $I$ is at most countable, such that
	\begin{gather}
	\nu=|u|^{q(\cdot)} + \sum_{i\in I}\nu_i\delta_{x_i},\label{T.ccp.form.nu}\\
	\mu \geq |\Delta u|^{p(\cdot)} + \sum_{i\in I} \mu_i \delta_{x_i},\label{T.ccp.form.mu}\\
	S \nu_i^{\frac{1}{p_2^\ast(x_i)}} \leq \mu_i^{\frac{1}{p(x_i)}}, \quad \forall i\in I,\label{T.ccp.nu_mu}
	\end{gather}
	where $\delta_{x_i}$ is the Dirac mass at $x_i$.
\end{theorem}
As an application of Theorem~\ref{Theo.ccp}, we investigate the existence and multiplicity of solutions to problem~\eqref{e1.1}. For this purpose, in what follows we further assume that 
\begin{itemize}
	\item[$(\mathcal P)$] There exists a vector $l\in\mathbb{R}^N\setminus\{0\}$ such that for any $x\in\Omega$, $\eta(t):=p(x + tl)$ is monotone for $t \in I_x:= \{t\in\R: \, x + tl\in\Omega\}$.
	\item[$(\mathcal Q)$] $p^+<q^-$.
	\item[$(\mathcal{M})$] $M: \R^+_0 \to \R$ is increasing and there exists $m_0>0$ such that $M(t) \geq m_0 = M(0)$ for all $t\in \R^+_0$. 
	
	\item[$(\mathcal{F}_0)$] There exist a constant $C_1>0$ and a function $\alpha \in C_+(\overline\Omega)$ such that $p^+<\alpha^-\leq \alpha(x) < p^\ast_2(x)$ for all $x\in \overline\Omega$ and 
	$$
	|f(x,t)| \leq C_1\left(1+|t|^{\alpha(x)-1}\right) \ \ \text{for a.e.} \ x\in\Omega\ \text{and all} \ t\in \R.$$
\end{itemize}
By a (weak) solution of problem~\eqref{e1.1}, we mean a function $u\in X$ such that
\begin{gather*}\label{Def.w-sol}
\notag\int_\Omega|\Delta u|^{p(x)-2}\Delta u\Delta v\,\diff x + M\left(\int_\Omega\frac{1}{p(x)}|\nabla u|^{p(x)}\,\diff x\right)\int_\Omega|\nabla u|^{p(x)-2}\nabla u\cdot\nabla v\,\diff x\\
-\lambda\int_\Omega f(x,u)v\,\diff x-\int_\Omega |u|^{q(x)-2}uv\,\diff x=0,\quad \forall v\in X.
\end{gather*}
We first investigate the multiplicity of solutions to problem~\eqref{e1.1} when the nonlinearity is of generalized concave-convex type. Precisely, we assume that 
\begin{itemize}
	\item[$(\mathcal{F}_1)$] 	$f(x,-t)=-f(x,t)$ \ for a.e. $x\in\Omega$ and all $t\in\R$.
	
	\item[$(\mathcal{F}_2)$] There exist $r\in C_+(\overline\Omega)$ with $r^+<p^-$ and positive constants $C_i$ ($i=\overline{2,6}$) such that 
	$$
	C_2|t|^{r(x)}\leq  q^- F(x,t) \leq f(x,t)t+C_3|t|^{r(x)}+C_4|t|^{p(x)}\ \ \text{and}\ \ F(x,t) \leq C_5|t|^{r(x)}+C_6|t|^{p(x)}
	$$
	for a.e. $x\in\Omega$ and all $t\in\R$, where $F(x,t) := \int_0^tf(x,s)\,\diff s$.
\end{itemize}
A typical example for $f$ fulfilling $(\mathcal{F}_0)-(\mathcal{F}_2)$ is a $p(\cdot)$-sublinear term $f(x,t)=|t|^{r(x)-2}t$ with $r^+<p^-$. A more general and remarkable example is $f(x,t)= c_1|t|^{r(x)-2}t\log^\kappa (e+|t|)+c_2|t|^{m(x)-2}t$ with $r^+<p^-$, $m\in C_+(\overline\Omega)$ satisfying $m(x)\leq p(x)$ for all $x\in\overline{\Omega }$, $c_1>0$, $c_2\geq 0$ and $\kappa\geq 0$. Our first existence result is the following.
\begin{theorem}\label{Theo.cc} {\rm \textbf{(Infinitely many solutions for the generalized concave-convex type problem)}}
	Let $(\mathcal{P})$, $(\mathcal{Q})$, $(\mathcal{C})$, $(\mathcal{M})$, $(\mathcal{F}_0)$, $(\mathcal{F}_1)$ and $(\mathcal{F}_2)$ hold. Then, there exists $\lambda_\ast>0$ such that for any $\lambda\in(0,\lambda_\ast)$, problem \eqref{e1.1} admits infinitely many solutions. Furthermore, if let $u_\lambda$ be one of these solutions, then it holds that 
	$$\lim_{\lambda \to 0^+}\|u_\lambda\| = 0.$$
\end{theorem}
Finally, we investigate the existence of a nontrivial solution to problem~\eqref{e1.1} when the nonlinearity is of $p(\cdot)$-superlinear type. Precisely, we assume that
\begin{itemize}
	\item[$(\mathcal{F}_3)$] $\lim_{t\to 0}\frac{f(x,t)}{|t|^{p(x)-1}} = 0$ uniformly in $x\in \Omega$.
	\item[$(\mathcal{F}_4)$] There exists $\theta \in (p^+, q^-)$ such that 
	$$
	0< \theta F(x,t) \leq f(x,t)t  \ \ \text{for a.e.} \ x\in\Omega\ \text{and all} \ t\in \R\setminus\{0\}.
	$$
\end{itemize}
The next theorem is our second existence result.
\begin{theorem}\label{Theo.sl} {\rm \textbf{(Existence result for the $p(\cdot)$-superlinear type problem)}}
	Let $(\mathcal{P})$, $(\mathcal{Q})$, $(\mathcal{C})$, $(\mathcal{M})$, $(\mathcal{F}_0)$, $(\mathcal{F}_3)$ and $(\mathcal{F}_4)$ hold. Then, there exists $\lambda^\ast>0$ such that for any $\lambda \geq \lambda^\ast$, problem \eqref{e1.1} admits a nontrivial solution $u^\lambda$. Moreover, it holds that $$\lim_{\lambda \to \infty}\|u^\lambda\| = 0.$$
\end{theorem}

\section{Preliminaries and Notations}\label{Pre}
In this section, we  briefly review fundamental properties of Lebesgue-Sobolev spaces with variable exponent. Let $\Omega$ be a bounded Lipschitz domain in $\mathbb{R}^N.$ For $m\in C_+(\overline\Omega)$ and a $\sigma$-finite, complete measure $\mu$ in $\overline{\Omega},$ define the variable
exponent Lebesgue space $L_\mu^{m(\cdot)}(\Omega)$ as
$$
L_\mu^{m(\cdot)}(\Omega) := \left \{ u : \Omega\to\mathbb{R}\  \hbox{is}\  \mu-\text{measurable},\ \int_\Omega |u(x)|^{m(x)} \;\diff\mu < \infty \right \},
$$
endowed with the Luxemburg norm
$$
\abs{u}_{L_\mu^{m(\cdot)}(\Omega)}:=\inf\left\{\lambda >0:
\int_\Omega
\Big|\frac{u(x)}{\lambda}\Big|^{m(x)}\;\diff\mu\le1\right\}.
$$
When $\diff \mu=\diff x$ the Lebesgue measure, as in Section~\ref{MainResults}, we write  $L^{m(\cdot) }(\Omega) $  and $\abs{u}_{m(\cdot)}$  in place of $L_\mu^{m(\cdot)}(\Omega)$ and $\abs{u}_{L_\mu^{m(\cdot)}(\Omega)}$, respectively. 

The following propositions are crucial for our arguments in the next sections.

\begin{proposition}[\cite{Diening}] \label{prop.Holder}
The space $L_{\mu}^{m(\cdot) }(\Omega )$ is a separable and uniformly convex Banach space, and its conjugate space is $L_{\mu}^{m'(\cdot) }(\Omega ),$ where  $1/m(x)+1/m'(x)=1$. For any $u\in L_{\mu}^{m(\cdot)}(\Omega)$ and $v\in L_{\mu}^{m'(\cdot)}(\Omega)$, we have
\begin{equation*}
\left|\int_\Omega uv\,\diff \mu\right|\leq\ 2 |u|_{L_\mu^{m(\cdot)}(\Omega)}|v|_{L_\mu^{m'(\cdot)}(\Omega)}.
\end{equation*}
\end{proposition}
Define the modular $\rho :L_{\mu}^{m(\cdot) }(\Omega )$ $ \to \mathbb{R}$ as
\[
\rho (u) =\int_{\Omega }| u(x)| ^{m(x) }\diff \mu,\quad
\forall u\in L_{\mu}^{m(\cdot) }(\Omega ) .
\]

\begin{proposition}[\cite{Diening}] \label{norm-modular}
For all $u\in L_\mu^{p(\cdot) }(\Omega ),$  we have
\begin{itemize}
\item[(i)] $|u|_{L_{\mu}^{m(\cdot)}(\Omega)}^{m^{+}}<1$ $(=1,>1)$
if and only if \  $\rho (u) <1$ $(=1,>1)$, respectively;

\item[(ii)] if \  $|u|_{L_{\mu}^{m(\cdot)}(\Omega)}^{m^{+}}>1,$ then  $|u|^{m^{-}}_{L_{\mu}^{m(\cdot)}(\Omega)}\leq \rho (u) \leq |u|_{L_{\mu}^{m(\cdot)}(\Omega)}^{m^{+}}$;
\item[(iii)] if \ $|u|_{L_{\mu}^{m(\cdot)}(\Omega)}<1,$ then $|u|_{L_{\mu}^{m(\cdot)}(\Omega)}^{m^{+}}\leq \rho
(u) \leq |u|_{L_{\mu}^{m(\cdot)}(\Omega)}^{m^{-}}$.
\end{itemize}
Consequently,
$$|u|_{L_{\mu}^{m(\cdot)}(\Omega)}^{m^{-}}-1\leq \rho (u) \leq |u|_{L_{\mu}^{m(\cdot)}(\Omega)}^{m^{+}}+1,\ \forall u\in L_\mu^{m(\cdot)}(\Omega ).$$
\end{proposition}
Thus, modular convergence and norm convergence on $L_{\mu}^{m(\cdot) }(\Omega )$ are equivalent.

\begin{proposition} \label{prop.eqiv.conv}
If $u,u_n\in L_{\mu}^{m(\cdot)}(\Omega) $ ($n=1,2,\cdots$), then the
following statements are equivalent:
\begin{itemize}
\item[(i)] $\lim_{n\to \infty }|u_n-u|_{L_{\mu}^{m(\cdot)}(\Omega)}=0$;

\item[(ii)] $\lim_{n\to \infty }\rho (u_n-u)=0$.

\end{itemize}
\end{proposition}

Let $k\in\N$ and define Sobolev spaces $W^{k,m(\cdot)}(\Omega)$ as in Section~\ref{MainResults}. We have the following crucial imbeddings on  $W^{k,m(\cdot)}(\Omega)$. 

\begin{proposition} [\cite{Diening}] \label{critical.imb}
Let $k\in\N$ and let $m\in C^{\log}_+(\overline{\Omega})$ be such that $k m^+<N$. Let $t\in C(\overline{\Omega })$ satisfy
$$1\leq  t(x)\leq m_k^\ast(x):=\frac{Nm(x)}{N-km(x)},\quad \forall x\in\overline{\Omega }.$$
Then, we obtain the continuous imbedding
\begin{center}
$W^{k,m(\cdot)}(\Omega) \hookrightarrow L^{t(\cdot) }(\Omega ).$
\end{center}
If we assume in addition that $t(x)< m_k^\ast(x)$ for all $x\in\overline{\Omega },$ then the above imbedding is compact.
\end{proposition}
\begin{proposition} [\cite{Zang-Fu}] \label{compact.imb.ZP}
	Assume that $m\in C^{\log}_+(\overline{\Omega})$ such that $2m^+<N.$ Then, we have the compact  imbedding
	\begin{center}
		$W^{2,m(\cdot)}(\Omega) \hookrightarrow \hookrightarrow W^{1,m(\cdot)}(\Omega).$
	\end{center}
	\end{proposition}

\section{Proof of The Concentration-Compactness Principle}\label{CCP}
In this section we give a proof of Theorem~\ref{Theo.ccp} by modifying the idea used in \cite {Bonder,Ho-Sim} that extended the concentration-compactness principle by Lions \cite{Lions} to the variable exponent case. Before giving a proof of Theorem~\ref{Theo.ccp}, we review some auxiliary results obtained in \cite{Bonder}.


\begin{lemma}[\cite{Bonder}] \label{L.convergence}
Let $\nu,\{\nu_n\}_{n\in\mathbb{N}}$ be nonnegative and finite Radon measures on $\overline{\Omega}$ such that $\nu_n\overset{\ast }{\rightharpoonup } \nu$ in $\mathcal{M}(\overline{\Omega})$. Then, for any $m\in C_{+}(\overline{\Omega })$,
$$
|\phi|_{L^{m(\cdot)}_{\nu_n}(\overline{\Omega})} \to
|\phi|_{L^{m(\cdot)}_{\nu}(\overline{\Omega})}, \quad \forall \phi\in C(\overline{\Omega}).$$
\end{lemma}

\begin{lemma}[\cite{Bonder}]\label{L.reserveHolder}
Let $\mu,\nu$ be two nonnegative and finite Borel measures
on $\overline{\Omega}$, such that there exists some  constant $C>0$ such that
$$
|\phi|_{L_\nu^{t(\cdot)}(\overline{\Omega})}\leq C|\phi|_{L_\mu^{s(\cdot)}(\overline{\Omega})},\ \ \forall \phi\in C^\infty(\overline{\Omega})
$$
for some $s,t\in C_+(\overline{\Omega})$ satisfying $s(x)<t(x)$ for all $x\in \overline{\Omega}$. Then, there exist an at most countable set $\{x_i\}_{i\in I}$ of distinct points in $\overline{\Omega}$ and
$\{\nu_i\}_{i\in I}\subset (0,\infty)$, such that
$$
\nu=\sum_{i\in I}\nu_i\delta_{x_i}.
$$
\end{lemma}
The following result is an extension of the Brezis-Lieb Lemma to variable exponent Lebesgue spaces. 
\begin{lemma}[\cite{Ho-Sim}]\label{L.brezis-lieb}
Let $\{f_n\}$ be a bounded sequence in $L^{t(\cdot)}(\Omega)$ ($t\in C_{+}(\overline{\Omega })$) and $f_n(x)\to f(x)$ a.e. $x\in\Omega$. Then $f\in L^{t(\cdot)}(\Omega)$ and
$$
\lim_{n\to\infty}\int_\Omega \left||f_n|^{t(x)}
-|f_n-f|^{t(x)}-|f|^{t(x)}\right|\diff x=0.
$$
\end{lemma}


\begin{proof}[\textbf{Proof of Theorem~\ref{Theo.ccp}}]

Let $v_n=u_n-u$. Then, up to a subsequence, we have
\begin{eqnarray}\label{T.conv.of.v_n}
\begin{cases}
v_n(x) &\to  \quad 0 \quad \text{a.e.}\quad  x\in\Omega,\\
v_n &\rightharpoonup \quad 0 \quad \text{in}\quad  X.
\end{cases}
\end{eqnarray}
So, by Lemma~\ref{L.brezis-lieb}, we deduce that
$$
\lim_{n\to\infty}\int_\Omega\left||u_n|^{q(x)}
-|v_n|^{q(x)}-|u|^{q(x)}\right|\diff x=0.$$
From this and \eqref{T.conv.of.v_n}, we easily obtain
$$
\lim_{n\to\infty}\Big(\int_\Omega \phi |u_n|^{q(x)}\diff x
-\int_\Omega\phi |v_n|^{q(x)} \diff x\Big)
=\int_\Omega \phi |u|^{q(x)} \diff x, \quad \forall  \phi\in C(\overline{\Omega}),$$
i.e.,
\begin{equation}\label{T.w*-vn}
|v_n|^{q(\cdot)}\overset{\ast }{\rightharpoonup }\bar{\nu}=\nu-|u|^{q(\cdot)}\quad \text{in} \ \ \mathcal{M}(\overline{\Omega}).
\end{equation}
It is clear that $\{|\Delta v_n|^{p(\cdot)}\}$ is bounded in $L^1(\Omega)$. So up to a subsequence, we have
\begin{equation}
|\Delta v_n|^{p(\cdot)} \overset{\ast }{\rightharpoonup }\quad \bar{\mu}\quad \text{in}\quad \mathcal{M}(\overline{\Omega})
\end{equation}
for some finite nonnegative Radon measure $\bar{\mu}$ on $\overline{\Omega}$. Clearly, $\phi v\in X$ for any $\phi\in C^\infty(\overline{\Omega})$ and for any $v\in X$. So, utilizing \eqref{S}, for any $\phi\in C^\infty(\overline{\Omega}),$ we get
\begin{align}
S|\phi v_n|_{q(\cdot)} &\leq |\Delta(\phi v_n)|_{p(\cdot)}\notag\\
&\leq |\phi\Delta v_n|_{p(\cdot)}+2|\nabla\phi\cdot\nabla v_n|_{p(\cdot)}+|v_n\Delta\phi|_{p(\cdot)}\notag\\
&\leq |\phi\Delta v_n|_{p(\cdot)}+2\|\phi\|_{C^2(\overline{\Omega})}\|v_n\|_{W^{1,p(\cdot)}(\Omega)}.\label{T.est.norm1}
\end{align}
Note that $v_n\to 0$ in $W^{1,p(\cdot)}(\Omega)$ in view of Proposition~\ref{compact.imb.ZP}. From this and Lemma~\ref{L.convergence}, we deduce from \eqref{T.w*-vn}-\eqref{T.est.norm1} that
$$
S|\phi|_{L_{\bar{\nu}}^{q(\cdot)}(\overline{\Omega})}\leq |\phi|_{L_{\bar{\mu}}^{p(\cdot)}(\overline{\Omega})},\quad \forall \phi\in C^\infty(\overline{\Omega}).$$
Thus, from Lemma~\ref{L.reserveHolder}, we obtain \eqref{T.ccp.form.nu}.
We claim that $\{x_i\}_{i\in I}\subset \mathcal{A}$. Assume by contradiction that there is some $x_i\in \overline{\Omega}\setminus\mathcal{A}$. Let $\delta>0$ be such that $\overline{B_{2\delta}(x_i)}\subset \mathbb{R}^N\setminus\mathcal{A}$. Set $B=B_\delta(x_i)\cap \overline{\Omega}$ then $\overline{B}\subset \overline{\Omega}\setminus\mathcal{A}$ and hence, $q(x)<p_2^\ast(x)$ for all $x\in\overline{B}$. Since $1<q(x)<p_2^\ast(x)$ for all $x\in\overline{B}\cap \overline{\Omega}$, we can find $\widetilde{q}\in C_+(\overline{\Omega})$ such that $\widetilde{q}|_{\overline{B}}=q$ and $\widetilde{q}(x)<p_2^\ast(x)$ for all $x\in\overline{\Omega}$. Thus, $u_n\to u$ in $L^{\widetilde{q}(\cdot)}(\Omega)$ in view of Proposition~\ref{critical.imb}. Equivalently, $\int_{\Omega}|u_n-u|^{\widetilde{q}(x)}\diff x\to 0;$ hence, $\int_{B}|u_n-u|^{q(x)}\diff x\to 0.$ This yields
$$\int_{B}|u_n|^{q(x)}\diff x\to \int_{B}|u|^{q(x)}\diff x$$
in view of Lemma~\ref{L.brezis-lieb}. From this and the fact that $\nu (B)\leq \liminf_{n\to \infty}\int_{B}|u_n|^{q(x)}\diff x$ (see \cite[Proposition 1.203]{Fonseca}), we obtain $\nu (B)\leq \int_{B}|u|^{q(x)}\diff x.$
Meanwhile, from \eqref{T.ccp.form.nu}, we have
$$\nu (B)\geq \int_{B}|u|^{q(x)}\diff x+\nu_i>\int_{B}|u|^{q(x)}\diff x,$$
a contradiction. So $\{x_i\}_{i\in I}\subset \mathcal{A}$.

Next, to obtain \eqref{T.ccp.nu_mu}, let $\eta$ be in $C_c^\infty(\mathbb{R}^N)$ such that $0\leq \eta\leq 1,$ $\eta\equiv 1$ on $B_{1/2}(0)$ and $\eta\equiv 0$ outside $B_1(0)$. Fix $i\in I.$ For $\epsilon>0,$ set $\Omega_{i,\epsilon}:=B_\epsilon(x_i)\cap \overline{\Omega},$ $\phi_{i,\epsilon}(x):=
\eta(\frac{x-x_i}{\epsilon})$ and
\begin{gather*}
p^+_{i,\epsilon}:=\sup_{x\in \Omega_{i,\epsilon}}p(x),\quad
p^-_{i,\epsilon}:=\inf_{x\in \Omega_{i,\epsilon}}p(x),\\
q^+_{i,\epsilon}:=\sup_{x\in \Omega_{i,\epsilon}}q(x),\quad
q^-_{i,\epsilon}:=\inf_{x\in \Omega_{i,\epsilon}}q(x).
\end{gather*}
Utilizing  \eqref{S} again, we have
\begin{align*}
S&|\phi_{i,\epsilon} u_n|_{q(\cdot)}\notag\\
&\leq \big|\phi_{i,\epsilon}\Delta u_n\big|_{p(\cdot)}+2\big|\nabla u_n\cdot\nabla \phi_{i,\epsilon}\big|_{p(\cdot)}+\big|u_n\Delta \phi_{i,\epsilon}\big|_{p(\cdot)}\notag\\
&\leq \big|\phi_{i,\epsilon}\Delta u_n\big|_{p(\cdot)}+2\big|\nabla u\cdot\nabla \phi_{i,\epsilon}\big|_{p(\cdot)}+\big|u\Delta \phi_{i,\epsilon}\big|_{p(\cdot)}+2\|\phi_{i,\epsilon}\|_{C^2(\overline{\Omega})}\|u_n-u\|_{W^{1,p(\cdot)}(\Omega)}.\label{T.etsphiun}
\end{align*}
Letting $n\to \infty$ in the last estimate with taking  Lemma~\ref{L.convergence} into account and noting $u_n\to u$ in $W^{1,p(\cdot)}(\Omega)$ due to Proposition~\ref{compact.imb.ZP}, we get that
\begin{equation}\label{T.estforsingular}
S |\phi_{i,\epsilon}|_{L^{q(\cdot)}_\nu(\overline{\Omega})} \leq
|\phi_{i,\epsilon}|_{L^{p(\cdot)}_\mu(\overline{\Omega})}+2\big|\nabla u\cdot\nabla \phi_{i,\epsilon}\big|_{p(\cdot)}+\big|u\Delta \phi_{i,\epsilon}\big|_{p(\cdot)}.
\end{equation}
Invoking Proposition~\ref{norm-modular}, we have 
$$|\phi_{i,\epsilon}|_{L^{q(\cdot)}_\nu(\overline{\Omega})}\geq \min\left\{\left(\int_{\Omega_{i,\epsilon}}|\phi_{i,\epsilon}|^{q(x)}\diff \nu\right)^{\frac{1}{q^+_{i,\epsilon}}} ,\left(\int_{\Omega_{i,\epsilon}}|\phi_{i,\epsilon}|^{q(x)}\diff\nu\right)^{\frac{1}{q^-_{i,\epsilon}}}\right\}.$$
Meanwhile,
$$\int_{\Omega_{i,\epsilon}}|\phi_{i,\epsilon}|^{q(x)}\diff\nu\geq
\int_{B_{\epsilon/2}(x_i)\cap \overline{\Omega}}|
\phi_{i,\epsilon}|^{q(x)}\diff\nu\geq \nu_i.$$
Thus,
\begin{equation}\label{T.nu_i}
|\phi_{i,\epsilon}|_{L^{q(\cdot)}_\nu(\overline{\Omega})}\geq \min\left\{\nu_i^{\frac{1}{q^+_{i,\epsilon}}} ,\nu_i^{\frac{1}{q^-_{i,\epsilon}}}\right\}.
\end{equation}
On the other hand, by applying Proposition~\ref{norm-modular} again, we have 
$$|\phi_{i,\epsilon}|_{L^{p(\cdot)}_\mu(\overline{\Omega})}\leq \max\left\{\left(\int_{\Omega_{i,\epsilon}}|\phi_{i,\epsilon}|^{p(x)}\diff\mu\right)^{\frac{1}{p^+_{i,\epsilon}}} ,\left(\int_{\Omega_{i,\epsilon}}|\phi_{i,\epsilon}|^{p(x)}\diff\mu\right)^{\frac{1}{p^-_{i,\epsilon}}}\right\}.$$
Meanwhile,
$$\int_{\Omega_{i,\epsilon}}|\phi_{i,\epsilon}|^{p(x)}\diff\mu\leq \int_{\Omega_{i,\epsilon}}\diff\mu= \mu(\Omega_{i,\epsilon}).$$
Thus, we obtain
\begin{equation}\label{T.mu_i}
|\phi_{i,\epsilon}|_{L^{p(\cdot)}_\mu(\overline{\Omega})}\leq \max\left\{\mu(\Omega_{i,\epsilon})^{\frac{1}{p^+_{i,\epsilon}}} ,\mu(\Omega_{i,\epsilon})^{\frac{1}{p^-_{i,\epsilon}}}\right\}.
\end{equation}
To estimate the remaining two terms in \eqref{T.estforsingular}, we first note that $|\nabla u|\in L^{p_1^\ast(\cdot)}(\Omega)$ (since $|\nabla u|\in W^{1,p(\cdot)}(\Omega)$) and $u\in L^{p_2^\ast(\cdot)}(\Omega)$ in view of Proposition~\ref{critical.imb}. Then using Proposition~\ref{prop.Holder}, we obtain
\begin{align}\label{est.critical.1}
\int_{\Omega} |\nabla u\cdot\nabla\phi_{i,\epsilon}|^{p(x)} \diff x&=\int_{\Omega_{i,\epsilon}} |\nabla u|^{p(x)}|\nabla\phi_{i,\epsilon}|^{p(x)} \diff x\notag\\
& \leq 2\big| |\nabla u|^{p(\cdot)}\big|_{L^{\frac{p_1^\ast(\cdot)}{p(\cdot)}}(\Omega_{i,\epsilon})}
\big||\nabla \phi_{i,\epsilon}|^{p(\cdot)} \big|_{L^{\frac{N}{p(\cdot)}}
(\Omega_{i,\epsilon})}
\end{align}
and 
\begin{align}\label{est.critical.2}
\int_{\Omega} |u\Delta\phi_{i,\epsilon}|^{p(x)} \diff x&=\int_{\Omega_{i,\epsilon}} |u|^{p(x)}|\Delta\phi_{i,\epsilon}|^{p(x)} \diff x\notag\\
& \leq 2\big|| u|^{p(\cdot)}\big|_{L^{\frac{p_2^\ast(\cdot)}{p(\cdot)}}(\Omega_{i,\epsilon})}
\big||\Delta \phi_{i,\epsilon}|^{p(\cdot)} \big|_{L^{\frac{N}{2p(\cdot)}}
	(\Omega_{i,\epsilon})}.
\end{align}
Applying Proposition~\ref{norm-modular}, we have 
$$
\big||\nabla \phi_{i,\epsilon}|^{p(\cdot)} \big|_{L^{\frac{N}{p(\cdot)}}
(\Omega_{i,\epsilon})}
\le \round{1+ \int_{B_\epsilon(x_i)} |\nabla \phi_{i,\epsilon}|^N \diff x}^{p^+/N}  =\round{ 1+\int_{B_1(0)} |\nabla\eta(y)|^N\, \diff y}^{p^+/N}
$$
and
$$
\big||\Delta \phi_{i,\epsilon}|^{p(\cdot)} \big|_{L^{\frac{N}{2p(\cdot)}}
	(\Omega_{i,\epsilon})}
\le \round{1+ \int_{B_\epsilon(x_i)} |\Delta \phi_{i,\epsilon}|^{N/2} \diff x}^{2p^+/N}  =\round{ 1+\int_{B_1(0)} |\Delta\eta(y)|^{N/2}\, \diff y}^{2p^+/N}.
$$
Utilizing the last two estimates, we deduce from \eqref{est.critical.1} and \eqref{est.critical.2} that
$$2\int_{\Omega} |\nabla u\cdot\nabla\phi_{i,\epsilon}|^{p(x)} \diff x+\int_{\Omega} |u\Delta\phi_{i,\epsilon}|^{p(x)} \diff x\to 0\ \ \text{as}\ \ \epsilon \to 0^+.$$
By Proposition~\ref{prop.eqiv.conv}, this is equivalent to
\begin{equation}\label{T.ugradphi}
2\big|\nabla u\cdot\nabla \phi_{i,\epsilon}\big|_{p(\cdot)}+\big|u\Delta \phi_{i,\epsilon}\big|_{p(\cdot)}\to 0\ \ \text{as}\ \ \epsilon \to 0^+.
\end{equation}
Letting $\epsilon \to 0^+$ in \eqref{T.estforsingular} and taking into account \eqref{T.nu_i}, \eqref{T.mu_i}, \eqref{T.ugradphi}, the continuity of $p,q$ on $\overline{\Omega},$ and the fact that $x_i\in  \mathcal{A}$ we obtain
$$S \nu_i^{1/p_2^\ast(x_i)} \le \mu_i^{1/p(x_i)},
$$
where $\mu_i := \mu(\{x_i\})$. In particular, $\{x_i\}_{i\in I}$ are atoms of $\mu$.

Finally, to show \eqref{T.ccp.form.mu}, note that for any $\phi\in C(\overline{\Omega})$  with $\phi\geq 0$, the functional $u\mapsto \int_{\Omega}\phi(x)|\Delta u|^{p(x)}\diff x$ is convex and differentiable on $X$. Hence, it is weakly lower semicontinuous and therefore,
$$\int_{\Omega}\phi(x)|\Delta u|^{p(x)}\diff x\leq \liminf_{n\to  \infty } \int_{\Omega }\phi(x)|\Delta u_n| ^{p(x)}\diff x=\int_{\overline{\Omega }}\phi \diff \mu.$$
Thus, $\mu\geq |\Delta u|^{p(\cdot)}$. Extract $\mu$ to its atoms, we deduce \eqref{T.ccp.form.mu}.
\end{proof}

\section{Proofs of Existence Results}\label{Existence}
In this section, we prove the existence results stated in Section~\ref{MainResults}, that is, Theorems~\ref{Theo.cc} and \ref{Theo.sl}. To determine solutions of problem \eqref{e1.1} we consider the energy functional associated with problem~\eqref{e1.1}:
\begin{align*}\label{J_lambda}
\notag J_\lambda(u) :=\int_\Omega &\frac{1}{p(x)}|\Delta u|^{p(x)} \diff x+\widehat{M}\left(\int_\Omega \frac{1}{p(x)}|\nabla u|^{p(x)} \diff x\right)\\
&-\lambda\int_\Omega F(x,u)\diff x-\int_\Omega \frac{1}{q(x)}|u|^{q(x)}\diff x,\quad u\in X,
\end{align*}
where $\widehat{M}(t):=\int_0^t M(s)\diff s$. Clearly, $J_\lambda: X\to\R$ is of class $C^1$ and any critical point of $J_\lambda$ is a solution to problem~\eqref{e1.1}. In order to prove Theorem~\ref{Theo.cc}, we apply a truncation technique used in \cite{Ber-Gar-Per.1996,Fig-Nas.EJDE2016} as follows. Fix $t_0>0$, which will be specific later, and define a truncation of $M(t)$ as
\begin{equation}\label{Def.M0}
M_0(t):=\begin{cases}
M(t),\ \ &0\leq t\leq t_0,\\
M(t_0),\ \ &t>t_0.
\end{cases}
\end{equation}
Then, we consider a modified energy functional $\widetilde{J}_\lambda: X\to\R$ defined as
\begin{align}\label{widetilde{J}_lambda}
\notag \widetilde{J}_\lambda(u):=\int_\Omega &\frac{1}{p(x)}|\Delta u|^{p(x)} \diff x+\widehat{M}_0\left(\int_\Omega \frac{1}{p(x)}|\nabla u|^{p(x)} \diff x\right)\\
&-\lambda\int_\Omega F(x,u) 
\diff x-\int_\Omega \frac{1}{q(x)}|u|^{q(x)} \diff x,\quad u\in X,
\end{align}
where $\widehat{M}_0(t):=\int_0^t M_0(s)\diff s$. It is clear that $M_0\in C(\R_0^+,\R^+)$, 
\begin{equation}\label{Est.M0}
m_0\leq M_0(t)\leq M(t_0),\quad \forall t\in\R_0^+
\end{equation}
and
\begin{equation}\label{Est.hat.M0}
m_0t\leq \widehat{M}_0(t)\leq M(t_0)t,\quad \forall t\in\R_0^+.
\end{equation}
The modified functional $\widetilde{J}_\lambda$ is also of class $C^1(X,\R)$ and its Fr\'echet derivative $\widetilde{J}_\lambda': X\to X^\ast$ is given by 
\begin{gather*}
\left\langle \widetilde{J}_\lambda' (u),v\right\rangle=\int_\Omega|\Delta u|^{p(x)-2}\Delta u\Delta v\,\diff x + M_0\left(\int_\Omega\frac{1}{p(x)}|\nabla u|^{p(x)}\,\diff x\right)\int_\Omega|\nabla u|^{p(x)-2}\nabla u\cdot\nabla v\,\diff x\\
-\lambda\int_\Omega f(x,u)v\,\diff x-\int_\Omega |u|^{q(x)-2}uv\,\diff x,\quad \forall\, u,v\in X.
\end{gather*}
Here and in the sequel, $X^\ast$ denotes the dual space of $X$ and $\langle \cdot,\cdot \rangle$ denote the duality pairing between $X$ and $X^\ast$.  Obviously, any critical point $u$ of $\widetilde{J}_\lambda$ with $\int_\Omega\frac{1}{p(x)}|\nabla u|^{p(x)}\,\diff x\leq t_0$ is a solution to problem~\eqref{e1.1}.

In the rest of this section, on $X$ we will make use of the following equivalent norm:
\begin{equation*}
\|u\|_0:=\inf\left\{\tau>0: \int_\Omega \frac{1}{p(x)}\left[\left|\frac{\Delta u}{\tau}\right|^{p(x)}+\left|\frac{\nabla u}{\tau}\right|^{p(x)}\right] \diff x\leq 1\right\}.
\end{equation*}
As a standard relation between norm and modular, we have
\begin{equation}\label{norm_0-modular}
\min\left\{\|u\|_0^{p^-},\|u\|_0^{p^+}\right\}\leq \int_\Omega \frac{1}{p(x)}\left[|\Delta u|^{p(x)} +|\nabla u|^{p(x)}\right] \diff x\leq \max\left\{\|u\|_0^{p^-},\|u\|_0^{p^+}\right\}
\end{equation}
for all $u\in X$. We also make use of the fact that assumption $(\mathcal P)$ to ensures
\begin{equation}\label{mu1.def}
\mu_1:=\inf_{\varphi\in C_c^\infty(\Omega)\setminus\{0\}}\frac{\int_\Omega|\nabla\varphi|^{p(x)}\diff x}{\int_\Omega|\varphi|^{p(x)}\diff x}>0
\end{equation}
that was shown in \cite{Fan-Zhang-Zhao.2005}. That is, under the assumption $(\mathcal P)$ we have
\begin{equation}\label{mu1}
\mu_1 \int_\Omega|\varphi|^{p(x)}\diff x\leq \int_\Omega|\nabla\varphi|^{p(x)}\diff x,\quad \forall \varphi\in X.
\end{equation}

\subsection{The generalized concave-convex type problem}
In this subsection we prove the existence of infinitely many solutions for problem~\eqref{e1.1} when the nonlinearity is a general form of concave-convex type. To this purpose, we always assume that the assumptions $(\mathcal{P})$, $(\mathcal{Q})$, $(\mathcal{C})$, $(\mathcal{M})$, $(\mathcal{F}_0)$, $(\mathcal{F}_1)$ and $(\mathcal{F}_2)$ hold. 

Now, we define the truncation $M_0(t)$ of $M(t)$ given in \eqref{Def.M0} and the truncated energy functional $\widetilde{J}_\lambda$ given in \eqref{widetilde{J}_lambda} by fixing $t_0\in (0,1)$ such that
\begin{equation*}\label{t0-sub}
m_0<M(t_0)<q^-m_0
\end{equation*} 
For simplicity of notation, we denote
\begin{equation}\label{lambda1}
\lambda_1:=\frac{q^-}{2p^+C_4}\left(m_0-\frac{M(t_0)}{q^-}\right)\mu_1>0
\end{equation}
and
\begin{equation}\label{K}
K:=\underset{*\in\{+,-\}}{\max}(l^*)^{-\frac{l^*}{l^*-1}}(l^*-1)a^{-\frac{1}{l^*-1}}b^{\frac{l^*}{l^*-1}}>0,
\end{equation}
with $l(x):=\frac{p(x)}{r(x)}>1$ for all $x\in\overline{\Omega}$,  $a:=	\frac{1}{2p^+}\left(m_0-\frac{M(t_0)}{q^-}\right)\mu_1>0$ and $b:= \frac{2C_3}{q^-}|1|_{\frac{l(\cdot)}{l(\cdot)-1}}>0$, where $\mu_1$ and $C_3,C_4$ are given by \eqref{mu1.def} and $(\mathcal{F}_2)$, respectively. The following local compactness is essential for seeking critical points of $\widetilde{J}_\lambda$.
\begin{lemma} \label{Le.PS1}
	Let $\lambda\in (0,\lambda_1)$ and let $\{u_n\}$ be a bounded sequence in $X$ such that
	\begin{equation}\label{Le.PS1.PS-seq}
	\widetilde{J}_\lambda(u_n)\to c\ \ \text{and}\ \ \widetilde{J}_\lambda'(u_n)\to 0
	\end{equation}
	for some $c\in\R$ satisfying
	\begin{align}\label{Le.PS1.c}
	c<\left(\frac{1}{p^+}-\frac{1}{q^-}\right)S^{\frac{N}{2}}-K\max\left\{\lambda^{\frac{l^+}{l^+-1}},\lambda^{\frac{l^-}{l^--1}}\right\},
	\end{align}
	where $S$ and $l,K$ are given by \eqref{S} and \eqref{K}, respectively. Then, $\{u_n\}$ has a convergent subsequence.
\end{lemma}
\begin{proof}
	Let $\lambda\in (0,\lambda_1)$ and let $\{u_n\}$ be a bounded in $X$ and satisfy \eqref{Le.PS1.PS-seq} with $c$ satisfying \eqref{Le.PS1.c}. By the reflexivity of $X$ and Theorem~\ref{Theo.ccp}, up to a subsequence we have
	\begin{eqnarray}
	u_n(x) &\to& u(x)  \quad \text{a.e.} \ \ x\in\Omega,\label{PL.PS1.a.e}\\
	u_n &\rightharpoonup& u  \quad \text{in} \  X,\label{PL.PS1.w-conv}\\
	|\Delta u_n|^{p(\cdot)} &\overset{\ast }{\rightharpoonup }&\mu \geq |\Delta u|^{p(\cdot)} + \sum_{i\in I} \mu_i \delta_{x_i} \ \text{in}\  \mathcal{M}(\overline{\Omega}),\label{PL.PS1.mu}\\
	|u_n|^{q(\cdot)}&\overset{\ast }{\rightharpoonup }&\nu=|u|^{q(\cdot)} + \sum_{i\in I}\nu_i\delta_{x_i} \ \text{in}\ \mathcal{M}(\overline{\Omega}),\label{PL.PS1.nu}\\
	S \nu_i^{\frac{1}{p_2^\ast(x_i)}} &\leq& \mu_i^{\frac{1}{p(x_i)}}, \ \forall i\in I.\label{PL.PS1.mu-nu}
	\end{eqnarray}
	We claim that $I = \emptyset.$ Suppose on the contrary that there exists $i\in I$. Let $\epsilon>0$ and define $\phi_{i,\epsilon}$ as in the proof of Theorem~\ref{Theo.ccp}. Clearly, $\{\phi_{i,\epsilon}u_n\}$ is bounded in $X$ and we have
	\begin{align*}
\notag	\int_{\Omega}&\phi_{i,\epsilon}|\Delta u_n|^{p(x)}\diff x-\int_{\Omega}\phi_{i,\epsilon}|u_n|^{q(x)}\diff x\\
\notag	=&\langle \widetilde{J}_\lambda'(u_n) ,\phi_{i,\epsilon}u_n \rangle+\lambda\int_{\Omega}\phi_{i,\epsilon}f(x,u_n)u_n\diff x-2\int_{\Omega}|\Delta u_n|^{p(x)-2}\Delta u_n\left(\nabla u_n\cdot\nabla \phi_{i,\epsilon}\right)\diff x\\
\notag	&-\int_{\Omega}|\Delta u_n|^{p(x)-2}\Delta u_n u_n\Delta \phi_{i,\epsilon}\diff x -M_0\left(\int_\Omega\frac{1}{p(x)}|\nabla u_n|^{p(x)}\,\diff x\right)\int_{\Omega}\phi_{i,\epsilon}|\nabla u_n|^{p(x)}\diff x\\
	& -M_0\left(\int_\Omega\frac{1}{p(x)}|\nabla u_n|^{p(x)}\,\diff x\right)\int_{\Omega}|\nabla u_n|^{p(x)-2}u_n\left(\nabla u_n\cdot\nabla \phi_{i,\epsilon}\right)\diff x.
	\end{align*}
	Thus, we have
	\begin{align}\label{PL.PS1.mu_i-nu_i.ineq}
	\notag\bigg|	\int_{\Omega}&\phi_{i,\epsilon}|\Delta u_n|^{p(x)}\diff x-\int_{\Omega}\phi_{i,\epsilon}|u_n|^{q(x)}\diff x\bigg|\\
	\notag	\leq&\left|\langle \widetilde{J}_\lambda'(u_n) ,\phi_{i,\epsilon}u_n \rangle\right|+\lambda\int_{\Omega}\phi_{i,\epsilon}|f(x,u_n)u_n|\diff x+2\int_{\Omega}|\Delta u_n|^{p(x)-1}|\nabla u_n||\nabla \phi_{i,\epsilon}|\diff x\\
	\notag	&+\int_{\Omega}|\Delta u_n|^{p(x)-1}|u_n||\Delta \phi_{i,\epsilon}|\diff x +M_0\left(\int_\Omega\frac{1}{p(x)}|\nabla u_n|^{p(x)}\,\diff x\right)\int_{\Omega}\phi_{i,\epsilon}|\nabla u_n|^{p(x)}\diff x\\
	& +M_0\left(\int_\Omega\frac{1}{p(x)}|\nabla u_n|^{p(x)}\,\diff x\right)\int_{\Omega}|\nabla u_n|^{p(x)-1}|u_n||\nabla \phi_{i,\epsilon}|\diff x.
	\end{align}
	We now estimate each term in the right-hand side of \eqref{PL.PS1.mu_i-nu_i.ineq} as $n\to\infty$ and then $\epsilon\to 0^+$. First, we recall that $\{u_n\}$ is bounded in $X$, namely, there exists $C>0$ such that
	\begin{equation}\label{PL.PS1.bound}
	\sup_{n\in \N} \int_\Omega \left[|\Delta u_n|^{p(x)} +|\nabla u_n|^{p(x)}\right]\, \diff x\leq C.
	\end{equation}
	Moreover, by Proposition~\ref{compact.imb.ZP} we deduce from \eqref{PL.PS1.w-conv} that
	\begin{equation}\label{PL.PS1.s-conv}
	u_n\to u \ \ \text{in}\ \ W^{1,p(\cdot)}(\Omega),\  L^{\alpha(\cdot)}(\Omega)\ \text{and}\ L^{r(\cdot)}(\Omega).
	\end{equation}
	In the rest of the proof, we will also denote by $C$ various positive constants independent of $n$ and $\epsilon$. From the boundedness of $\{\phi_{i,\epsilon}u_n\}$ in $X$ and \eqref{Le.PS1.PS-seq}, we have
	\begin{equation}\label{PL.PS1.mu_i-nu_i.1}
	\limsup_{\epsilon\to 0^+}\,	\limsup_{n\to\infty} \left|\langle \widetilde{J}_\lambda'(u_n) ,\phi_{i,\epsilon}u_n \rangle\right|=0.
	\end{equation}
By $(\mathcal{F}_0)$, we deduce from \eqref{PL.PS1.a.e} and \eqref{PL.PS1.s-conv} that
	\begin{equation}\label{PL.PS1.mu_i-nu_i.2}
	\limsup_{\epsilon\to 0^+}\,	\limsup_{n\to\infty}\int_{\Omega}\phi_{i,\epsilon}|f(x,u_n)u_n|\diff x=\limsup_{\epsilon\to 0^+}\,\int_{\Omega}\phi_{i,\epsilon}|f(x,u)u|\diff x=0
	\end{equation}
in view of the Lebesgue dominated convergence theorem. For the third term, we estimate as follows. Let $\delta>0$ be arbitrary and fixed. By Young's inequality and \eqref{PL.PS1.bound} we have 
\begin{align*}
\notag\int_{\Omega}|\Delta u_n|^{p(x)-1}|\nabla u_n||\nabla \phi_{i,\epsilon}|\diff x&\leq \delta\int_{\Omega}|\Delta u_n|^{p(x)}\diff x+C(\delta)\int_{\Omega}|\nabla u_n|^{p(x)}|\nabla \phi_{i,\epsilon}|^{p(x)}\diff x\\
&\leq C\delta+C(\delta)\int_{\Omega}|\nabla u_n|^{p(x)}|\nabla \phi_{i,\epsilon}|^{p(x)}\diff x,
\end{align*}
where $C(\delta)$ denotes a positive constant depending $\delta$ but independent of $n$ and $\epsilon$. Combining this with \eqref{PL.PS1.s-conv} gives
\begin{equation*}
\limsup_{n\to\infty}\int_{\Omega}|\Delta u_n|^{p(x)-1}|\nabla u_n||\nabla \phi_{i,\epsilon}|\diff x\leq C\delta+C(\delta)\int_{\Omega}|\nabla u|^{p(x)}|\nabla \phi_{i,\epsilon}|^{p(x)}\diff x.
\end{equation*}
Then, arguing as that obtained \eqref{T.ugradphi} we derive
\begin{equation*}
\limsup_{\epsilon\to 0^+}\,	\limsup_{n\to\infty}\int_{\Omega}|\Delta u_n|^{p(x)-1}|\nabla u_n||\nabla \phi_{i,\epsilon}|\diff x\leq C\delta.
\end{equation*}
Since $\delta>0$ was taken arbitrarily we arrive at
\begin{equation}\label{PL.PS1.mu_i-nu_i.3}
\limsup_{\epsilon\to 0^+}\,	\limsup_{n\to\infty}\int_{\Omega}|\Delta u_n|^{p(x)-1}|\nabla u_n||\nabla \phi_{i,\epsilon}|\diff x=0.
\end{equation}
In a similar manner we have
\begin{equation}\label{PL.PS1.mu_i-nu_i.4}
\limsup_{\epsilon\to 0^+}\,	\limsup_{n\to\infty}\int_{\Omega}|\Delta u_n|^{p(x)-1}|u_n||\Delta \phi_{i,\epsilon}|\diff x=0.
\end{equation}
Using \eqref{Est.M0}, \eqref{PL.PS1.s-conv} and applying Proposition~\ref{prop.eqiv.conv}, we have
\begin{align*}
\notag \limsup_{n\to\infty}\, M_0\left(\int_\Omega\frac{1}{p(x)}|\nabla u_n|^{p(x)}\,\diff x\right)\int_{\Omega}\phi_{i,\epsilon}|\nabla u_n|^{p(x)}\diff x\leq &M(t_0)\int_{\Omega}\phi_{i,\epsilon}|\nabla u|^{p(x)}\diff x.
\end{align*}
It follows that
\begin{equation}\label{PL.PS1.mu_i-nu_i.5}
\limsup_{\epsilon\to 0^+}\,	\limsup_{n\to\infty}M_0\left(\int_\Omega\frac{1}{p(x)}|\nabla u_n|^{p(x)}\,\diff x\right)\int_{\Omega}\phi_{i,\epsilon}|\nabla u_n|^{p(x)}\diff x=0.
\end{equation}
Finally, invoking \eqref{Est.M0}, Young's inequality and \eqref{PL.PS1.bound} we deduce that  for an arbitrary $\delta>0$,
\begin{align*}
\notag M_0\bigg(\int_\Omega&\frac{1}{p(x)}|\nabla u_n|^{p(x)}\,\diff x\bigg)\int_{\Omega}|\nabla u_n|^{p(x)-1}|u_n||\nabla \phi_{i,\epsilon}|\diff x\\
&\leq M(t_0)\left[C\delta+C(\delta)\int_{\Omega}|u_n|^{p(x)}|\nabla \phi_{i,\epsilon}|^{p(x)}\diff x\right].
\end{align*}
From this and \eqref{PL.PS1.s-conv} we obtain
\begin{align*}
\limsup_{n\to\infty} \, M_0\bigg(\int_\Omega\frac{1}{p(x)}|\nabla u_n|^{p(x)}\,\diff x\bigg)&\int_{\Omega}|\nabla u_n|^{p(x)-1}|u_n||\nabla \phi_{i,\epsilon}|\diff x\\
&\leq M(t_0)\left[C\delta+C(\delta)\int_{\Omega}|u|^{p(x)}|\nabla \phi_{i,\epsilon}|^{p(x)}\diff x\right].
\end{align*}
Using estimation
$$\int_{\Omega} |u|^{p(x)}|\nabla\phi_{i,\epsilon}|^{p(x)} \diff x\leq 2\big| |u|^{p(\cdot)}\big|_{L^{\frac{p_1^\ast(\cdot)}{p(\cdot)}}(\Omega_{i,\epsilon})}
\big||\nabla \phi_{i,\epsilon}|^{p(\cdot)} \big|_{L^{\frac{N}{p(\cdot)}}
	(B_\epsilon(x_i))},$$
where $\Omega_{i,\epsilon}:=B_\epsilon(x_i)\cap\overline{\Omega}$, and then arguing as that obtained \eqref{T.ugradphi} we infer
\begin{equation}\label{PL.PS1.mu_i-nu_i.6}
\limsup_{\epsilon\to 0^+}\,	\limsup_{n\to\infty}\, M_0\bigg(\int_\Omega\frac{1}{p(x)}|\nabla u_n|^{p(x)}\,\diff x\bigg)\int_{\Omega}|\nabla u_n|^{p(x)-1}|u_n||\nabla \phi_{i,\epsilon}|\diff x=0.
\end{equation}
Taking limit superior as $n\to\infty$ and then as $\epsilon\to 0^+$ in \eqref{PL.PS1.mu_i-nu_i.ineq}, utilizing \eqref{PL.PS1.mu_i-nu_i.1}-\eqref{PL.PS1.mu_i-nu_i.6}, we arrive at
\begin{equation*}
\mu_i=\nu_i.
\end{equation*}
From this and \eqref{PL.PS1.mu-nu} we obtain
	\begin{equation}\label{PL.PS1.estmu_i}
	\nu_i=\mu_i\geq S^{\frac{N}{2}}.
	\end{equation}
	On the other hand, by the boundedness of $\{u_n\}$ and \eqref{Le.PS1.PS-seq} we have
	\begin{align*}
	c+o_n(1)=&\widetilde{J}_\lambda(u_n)-\frac{1}{q^-}\left\langle\widetilde{J}_\lambda'(u_n) ,u_n\right\rangle\\
	\geq& \left(\frac{1}{p^+}-\frac{1}{q^-}\right)\int_\Omega |\Delta u_n|^{p(x)} \diff x+\widehat{M}_0\left(\int_\Omega \frac{1}{p(x)}|\nabla u_n|^{p(x)} \diff x\right)\\
	&-\frac{1}{q^-} M_0\bigg(\int_\Omega\frac{1}{p(x)}|\nabla u_n|^{p(x)}\,\diff x\bigg)\int_{\Omega}|\nabla u_n|^{p(x)}\diff x\\
	&- \frac{\lambda}{q^-}\int_{\Omega}[q^-F(x,u_n)-f(x,u_n)u_n]\diff x.
\end{align*}
From this, \eqref{Est.M0}, \eqref{Est.hat.M0} and $(\mathcal{F}_2)$ we obtain
	\begin{align*}
	c+o_n(1)\geq&\left(\frac{1}{p^+}-\frac{1}{q^-}\right)\int_\Omega |\Delta u_n|^{p(x)} \diff x+\frac{1}{p^+}\left(m_0-\frac{M(t_0)}{q^-}\right) \int_{\Omega}|\nabla u_n|^{p(x)}\diff x\\
	&-\frac{\lambda }{q^-}\int_{\Omega}\big[C_3|u_n|^{r(x)}+C_4|u_n|^{p(x)}\big]\diff x.
	\end{align*}
Then, using \eqref{mu1} we have
	\begin{align*}
	c+o_n(1)\geq&\left(\frac{1}{p^+}-\frac{1}{q^-}\right)\int_\Omega |\Delta u_n|^{p(x)} \diff x+\frac{1}{p^+}\left(m_0-\frac{M(t_0)}{q^-}\right)\mu_1 \int_{\Omega}|u_n|^{p(x)}\diff x\\
	&-\frac{\lambda }{q^-}\int_{\Omega}\big[C_3|u_n|^{r(x)}+C_4|u_n|^{p(x)}\big]\diff x.
	\end{align*}
Note that $\lambda\in (0,\lambda_1)$ implies $\frac{\lambda C_4}{q^-}< \frac{1}{2p^+}\left(m_0-\frac{M(t_0)}{q^-}\right)\mu_1$ and hence, the last inequality yields
	\begin{align*}
c+o_n(1)\geq&\left(\frac{1}{p^+}-\frac{1}{q^-}\right)\int_\Omega |\Delta u_n|^{p(x)} \diff x+\frac{1}{2p^+}\left(m_0-\frac{M(t_0)}{q^-}\right)\mu_1 \int_{\Omega}|u_n|^{p(x)}\diff x\\
&-\frac{\lambda C_3 }{q^-}\int_{\Omega}|u_n|^{r(x)}\diff x.
\end{align*}

 Passing to the limit as $n\to\infty$ in the last inequality, taking \eqref{PL.PS1.mu}, \eqref{PL.PS1.s-conv} and \eqref{PL.PS1.estmu_i} into account, we infer
\begin{equation*}
c\geq\,  \left(\frac{1}{p^+}-\frac{1}{q^-}\right)S^{\frac{N}{2}}+\frac{1}{2p^+}\left(m_0-\frac{M(t_0)}{q^-}\right)\mu_1 \int_{\Omega}|u|^{p(x)}\diff x-\frac{\lambda C_3}{q^-}\int_{\Omega}|u|^{r(x)}\diff x.
\end{equation*}
By Proposition~\ref{prop.Holder}, we have
$$\int_{\Omega}|u|^{r(x)}\diff x\leq 2|1|_{\frac{l(\cdot)}{l(\cdot)-1}}\big||u|^{r(\cdot)}\big|_{l(\cdot)},$$
where $l(x):=\frac{p(x)}{r(x)}$. Then, it follows from the last two estimates that 
\begin{equation}\label{PL.PS1.c}
c\geq\,  \left(\frac{1}{p^+}-\frac{1}{q^-}\right)S^{\frac{N}{2}}+a \int_{\Omega}|u|^{p(x)}\diff x-\lambda b \big||u|^{r(\cdot)}\big|_{l(\cdot)},
\end{equation}
where $a:=	\frac{1}{2p^+}\left(m_0-\frac{M(t_0)}{q^-}\right)\mu_1$ and $b:= \frac{2C_3}{q^-}|1|_{\frac{l(\cdot)}{l(\cdot)-1}}$. We consider the following cases.
	\begin{itemize}
		\item Case $\big||u|^{r(\cdot)}\big|_{l(\cdot)}\geq 1$. Then \eqref{PL.PS1.c} and Proposition~\ref{norm-modular} yield
	\end{itemize}
	$$c\geq\,  \left(\frac{1}{p^+}-\frac{1}{q^-}\right)S^{\frac{N}{2}}+a \xi^{l^-}- b \lambda \xi=:h_1(\xi)\ \text{with}\ \xi:=\big||u|^{r(\cdot)}\big|_{l(\cdot)}\geq 1.$$
	Thus,
	$$c\geq \underset{\xi\geq 0}{\min}\ h_1(\xi)=h_1\left(\bigg(\frac{b\lambda}{a l^-}\bigg)^{\frac{1}{l^--1}}\right),$$
	i.e.,
	$$c\geq \left(\frac{1}{p^+}-\frac{1}{q^-}\right)S^{\frac{N}{2}}-(l^-)^{-\frac{l^-}{l^--1}}(l^--1)a^{-\frac{1}{l^--1}}b^{\frac{l^-}{l^--1}}\lambda^{\frac{l^-}{l^--1}}.$$
	\begin{itemize}
		\item Case $\big||u|^{r(\cdot)}\big|_{l(\cdot)}< 1$. Then \eqref{PL.PS1.c} and Proposition~\ref{norm-modular} yield
	\end{itemize}
	$$c\geq\,  \left(\frac{1}{p^+}-\frac{1}{q^-}\right)S^{\frac{N}{2}}+a \xi^{l^+}- b \lambda \xi=:h_2(\xi)\ \text{with}\ 0\leq \xi:=\big||u|^{r(\cdot)}\big|_{l(\cdot)}<1.$$
	Thus,
	$$c\geq \underset{\xi\geq 0}{\min}\ h_2(\xi)=h_2\left(\bigg(\frac{b\lambda}{a l^+}\bigg)^{\frac{1}{l^+-1}}\right),$$
	i.e.,
	$$c\geq \left(\frac{1}{p^+}-\frac{1}{q^-}\right)S^{\frac{N}{2}}-(l^+)^{-\frac{l^+}{l^+-1}} (l^+-1)a^{-\frac{1}{l^+-1}}b^{\frac{l^+}{l^+-1}}\lambda^{\frac{l^+}{l^+-1}}.$$
	Therefore, in any case, we obtain
	$$c\geq \left(\frac{1}{p^+}-\frac{1}{q^-}\right)S^{\frac{N}{2}}-K\max\left\{\lambda^{\frac{l^+}{l^+-1}},\lambda^{\frac{l^-}{l^--1}}\right\},$$
	where $K:=\underset{*\in\{+,-\}}{\max}(l^*)^{-\frac{l^*}{l^*-1}}(l^*-1)a^{-\frac{1}{l^*-1}}b^{\frac{l^*}{l^*-1}}$.
	However, this is in contradiction with \eqref{Le.PS1.c}. That is, we have shown that  $I=\emptyset$ and hence, \eqref{PL.PS1.nu} yields  $\int_{\Omega}|u_n|^{q(x)}\diff x\to \int_{\Omega}|u|^{q(x)}\diff x.$ Then, invoking \eqref{PL.PS1.a.e} and Lemma~\ref{L.brezis-lieb} we obtain
	\begin{equation*}
	u_n\to u\ \ \text{in}\ \ L^{q(\cdot)}(\Omega).
	\end{equation*}
	Combining this with \eqref{Le.PS1.PS-seq}, \eqref{PL.PS1.w-conv}, and \eqref{PL.PS1.s-conv}, invoking H\"older type inequality (Proposition~\ref{prop.Holder}), we obtain
	\begin{align*}
	\int_\Omega &\left(|\Delta u_n|^{p(x)-2}\Delta u_n-|\Delta u|^{p(x)-2}\Delta u\right)\left(\Delta u_n-\Delta u\right)\diff x\\
	=&\langle \widetilde{J}_\lambda' (u_n),u_n-u\rangle- M_0\left(\int_\Omega\frac{1}{p(x)}|\nabla u_n|^{p(x)}\,\diff x\right)\int_\Omega|\nabla u_n|^{p(x)-2}\nabla u_n\cdot\nabla (u_n-u)\,\diff x\\
	&+\lambda\int_\Omega f(x,u_n)(u_n-u)\,\diff x+\int_\Omega |u_n|^{q(x)-2}u_n(u_n-u)\,\diff x\\
	&-\int_\Omega |\Delta u|^{p(x)-2}\Delta u\left(\Delta u_n-\Delta u\right)\diff x\to 0,
	\end{align*}
	and hence, we derive $u_n\to u$ in $X$ via a standard argument (see, for example, \cite[Proof of Lemma 3.4]{Chung}). The proof is complete.
\end{proof} 

By Proposition~\ref{critical.imb}, there exists a constant $C_7>1$ such that
\begin{equation}\label{norms-norm}
\max\left\{|u|_{r(\cdot)},|u|_{q(\cdot)}\right\}\leq C_7\|u\|_0,\quad \forall \, u\in X.
\end{equation}
Set 
\begin{equation}\label{lambda2-k0}
k_0:=\min\{1,m_0\}\  \ \text{and}\ \ \lambda_2:=\frac{k_0 \mu_1}{2p^+C_6},
\end{equation}
where $C_6$ is given by $(\mathcal{F}_2)$. Let $\lambda\in(0,\lambda_2)$. By \eqref{mu1} we have
\begin{equation*}
\lambda C_6\int_\Omega |u|^{p(x)}\diff x\leq \lambda C_6\mu_1^{-1}\int_\Omega |\nabla u|^{p(x)}\diff x\leq \frac{k_0}{2}\int_\Omega \frac{1}{p(x)}|\nabla u|^{p(x)}\diff x,\quad \forall u\in X.
\end{equation*}
Combining this with \eqref{Est.hat.M0} and $(\mathcal{F}_2)$ gives
\begin{align*}
\notag \widetilde{J}_\lambda(u)&\geq \int_\Omega \frac{1}{p(x)}|\Delta u|^{p(x)} \diff x+m_0\int_\Omega \frac{1}{p(x)}|\nabla u|^{p(x)} \diff x-\lambda C_5\int_\Omega |u|^{r(x)}\diff x\\
&\quad-\lambda C_6\int_\Omega |u|^{p(x)}\diff x-\int_\Omega \frac{1}{q(x)}|u|^{q(x)} \diff x\\
&\geq \frac{k_0}{2}\int_\Omega \frac{1}{p(x)}\left[|\Delta u|^{p(x)}+|\nabla u|^{p(x)}\right] \diff x-\lambda C_5\int_\Omega |u|^{r(x)}\diff x-\frac{1}{q^-}\int_\Omega |u|^{q(x)} \diff x,\quad \forall u\in X.
\end{align*}
By invoking Proposition~\ref{norm-modular} then using \eqref{norms-norm} and \eqref{norm_0-modular} we deduce from the last inequality that
\begin{align}
\notag\widetilde{J}_\lambda(u)&\geq \frac{k_0}{2}\min\left\{\|u\|_0^{p^-},\|u\|_0^{p^+}\right\}-\lambda C_5\max\left\{|u|_{r(\cdot)}^{r^+},|u|_{r(\cdot)}^{r^-}\right\}-\frac{1}{q^-}\max\left\{|u|_{q(\cdot)}^{q^+},|u|_{q(\cdot)}^{q^-}\right\}\\
\notag&\geq \frac{k_0}{2}\min\left\{\|u\|_0^{p^-},\|u\|_0^{p^+}\right\}-\lambda C_5 C_7^{r^+}\max\left\{\|u\|_0^{r^-},\|u\|_0^{r^+}\right\}-\frac{C_7^{q^+}}{q^-}\max\left\{\|u\|_0^{q^-},\|u\|_0^{q^+}\right\}.
\end{align}
Thus,
 \begin{equation}\label{PTcc.gi}
 \widetilde{J}_\lambda(u)\geq  g_\lambda(\|u\|_0)\ \ \text{for} \ \ \|u\|_0\leq1,
  \end{equation} 
 where $g_\lambda\in C(\R_0^+)$ is given by
 $$g_\lambda(t):=\frac{k_0}{2}t^{p^+}-\lambda C_5 C_7^{r^+}t^{r^-}-\frac{C_7^{q^+}}{q^-}t^{q^-}.$$
Rewrite $g_\lambda(t)=C_5 C_7^{r^+}t^{r^-}\left(h(t)-\lambda\right)$ with $$h(t):=a_0t^{p^+-r^-}-b_0t^{q^--r^-},$$ 
where $a_0:=k_0(2C_5C_7^{r^+})^{-1}>0$ and $b_0:=(q^-C_5)^{-1}C_7^{q^+-r^+}>0$. Clearly,
\begin{align}\label{lambda3}
\notag\lambda_3:&=\max_{t\geq 0}\, h(t)=h\left(\left[\frac{(p^+-r^-)a_0}{(q^--r^-)b_0}\right]^{\frac{1}{q^--p^+}}\right)\\
&=a_0^{\frac{q^--r^-}{q^--p^+}}b_0^{\frac{r^--p^+}{q^--p^+}}\left(\frac{p^+-r^-}{q^--r^-}\right)^{\frac{p^+-r^-}{q^--p^+}}\frac{q^--p^+}{q^--r^-}>0
\end{align}
and for any $\lambda\in (0,\lambda_3)$, $g_\lambda(t)$ has only positive roots $t_1=t_1(\lambda)$ and $t_2=t_2(\lambda)$ with $$0<t_1<\left[\frac{(p^+-r^-)a_0}{(q^--r^-)b_0}\right]^{\frac{1}{q^--p^+}}=:t_*<t_2.$$
Obviously, on $\R_0^+$ the function $g_\lambda(t)$ is only negative  on $(0,t_1)\cup (t_2,\infty)$. Moreover, we have
\begin{equation}\label{lim.t1}
\lim_{\lambda\to 0^+}t_1(\lambda)=0.
\end{equation}
By \eqref{lim.t1}, we find $\lambda_4>0$ such that
\begin{equation}\label{t_1}
t_1<\min\left\{2^{\frac{-1}{p^-}},\left(2^{-1}t_*^{p^+}\right)^{1/p^-},t_0^{1/p^-}\right\},\quad \forall \lambda\in (0,\lambda_4).
\end{equation}
Set 
\begin{equation}\label{lambda_*^{(1)}}
\lambda_\ast^{(1)}:=\min\left\{\lambda_1,\lambda_2,\lambda_3,\lambda_4\right\},
\end{equation} 
where $\lambda_1,\lambda_2,\lambda_3$ and $\lambda_4$ are given by \eqref{lambda1}, \eqref{lambda2-k0}, \eqref{lambda3} and \eqref{t_1}, respectively. For each $\lambda\in \left(0,\lambda_\ast^{(1)}\right)$, we consider the truncated functional $T_\lambda: X\to\R$ given by
\begin{align*}\label{T_lambda}
\notag T_\lambda(u):=&\int_\Omega \frac{1}{p(x)}|\Delta u|^{p(x)} \diff x+\widehat{M}_0\left(\int_\Omega \frac{1}{p(x)}|\nabla u|^{p(x)} \diff x\right)\\
&-\phi\left(\int_\Omega \frac{1}{p(x)}\left[|\Delta u|^{p(x)} +|\nabla u|^{p(x)}\right] \diff x\right)\left[\lambda\int_\Omega F(x,u) 
\diff x+\int_\Omega \frac{1}{q(x)}|u|^{q(x)} \diff x\right]
\end{align*}
for $u\in X$, where $\phi\in C_c^\infty(\R)$, $0\leq \phi(t)\leq 1$ for all $t\in\R$, $\phi(t)=1$ for $|t|\leq t_1^{p^-}$ and $\phi(t)=0$ for $|t|\geq 2t_1^{p^-}$. Clearly, $T_\lambda\in C^1(X,\R)$ and it holds that
\begin{equation}\label{T_lambda.Est1}
T_\lambda(u)\geq \widetilde{J}_\lambda(u),\quad \forall u\in X,
\end{equation}
\begin{equation}\label{T_lambda.Est2}
T_\lambda(u)=\widetilde{J}_\lambda(u) \ \ \text{for all} \ \ u\in X \ \ \text{with} \ \ \int_\Omega \frac{1}{p(x)}\left[|\Delta u|^{p(x)} +|\nabla u|^{p(x)}\right] \diff x<t_1^{p^-},
\end{equation}
and
\begin{equation}\label{T_lambda.Est3}
T_\lambda(u)=\int_\Omega \frac{1}{p(x)}|\Delta u|^{p(x)} \diff x+\widehat{M}_0\left(\int_\Omega \frac{1}{p(x)}|\nabla u|^{p(x)} \diff x\right)
\end{equation}
  for all $u\in X$ with $\int_\Omega \frac{1}{p(x)}\left[|\Delta u|^{p(x)} +|\nabla u|^{p(x)}\right] \diff x>2t_1^{p^-}$.
\begin{lemma}\label{T_lambda(u)<0}
	Let $\lambda\in \left(0,\lambda_\ast^{(1)}\right)$. Then, we have $\int_\Omega \frac{1}{p(x)}\left[|\Delta u|^{p(x)} +|\nabla u|^{p(x)}\right] \diff x<t_1^{p^-}$ and hence, $T_\lambda(u)=\widetilde{J}_\lambda(u)$ whenever $T_\lambda(u)<0$.
\end{lemma}
\begin{proof}
	Suppose that $T_\lambda(u)<0$. Then, we have $\widetilde{J}_\lambda(u)<0$ due to \eqref{T_lambda.Est1}. We claim that $\|u\|_0\leq 1$. Indeed, suppose on the contrary that $\|u\|_0>1$. Then, $\int_\Omega \frac{1}{p(x)}\big[|\Delta u|^{p(x)} +|\nabla u|^{p(x)}\big] \diff x>1>2t_1^{p^-}$ due to \eqref{norm_0-modular} and \eqref{t_1} and therefore, \eqref{T_lambda.Est3} yields
	\begin{equation*}
\int_\Omega \frac{1}{p(x)}|\Delta u|^{p(x)} \diff x+\widehat{M}_0\left(\int_\Omega \frac{1}{p(x)}|\nabla u|^{p(x)} \diff x\right)<0,
	\end{equation*}
a contradiction. Thus, we have $\|u\|_0\leq 1$ and hence, it follows from \eqref{PTcc.gi} that $g_\lambda(\|u\|_0)\leq \widetilde{J}_\lambda(u)<0$. Thus, $\|u\|_0<t_1$ or $\|u\|_0>t_2>t_*$ due to \eqref{norm_0-modular} and \eqref{t_1}. If $\|u\|_0>t_*$, then $\int_\Omega \frac{1}{p(x)}\big[|\Delta u|^{p(x)} +|\nabla u|^{p(x)}\big] \diff x>\|u\|_0^{p^+}>t_*^{p^+}>2t_1^{p^-}$ due to \eqref{norm_0-modular} and \eqref{t_1}; hence, $T_\lambda(u)\geq 0$ due to \eqref{T_lambda.Est3}, which is absurd. Thus, $\|u\|_0<t_1$ and hence, $\int_\Omega \frac{1}{p(x)}\left[|\Delta u|^{p(x)} +|\nabla u|^{p(x)}\right] \diff x<t_1^{p^-}$. The proof is complete.

\end{proof}
Next, we will show that $T_\lambda$ admits a sequence of critical points $\{u_n\}$ with $\int_\Omega \frac{1}{p(x)}\big[|\Delta u_n|^{p(x)} +|\nabla u_n|^{p(x)}\big] \diff x<t_1^{p^-}$via genus theory. Let us denote by $\gamma (A)$ the genus of a closed subset $A\subset X\setminus\{0\}$ that is symmetric with respect to the origin, that is, $u\in A$ implies $-u\in A$ (see \cite{Rab1986} for a review of the definition and properties of the genus).
\begin{lemma}\label{genus.k}
	Let $\lambda\in \left(0,\lambda_\ast^{(1)}\right)$. Then, for each $k \in \mathbb{N}$, there exists $\epsilon>0$
	such that
	$$
	\gamma(T_\lambda^{-\epsilon}) \geq k,
	$$
	where $T_\lambda^{-\epsilon}:=\{u \in X: T_\lambda(u) \leq -\epsilon\}$. \end{lemma}
Note that by $(\mathcal{F}_1)$ and the definition of $T_\lambda$, $T_\lambda^{-\epsilon}$ is a closed subset of $X\setminus\{0\}$  and is symmetric with respect to the origin; hence, $\gamma(T_\lambda^{-\epsilon})$ is well defined.
\begin{proof}[Proof of Lemma~\ref{genus.k}]
	Let $k \in \mathbb{N}$ and let $X_{k}$ be a subspace of $X$ of dimension $k$. Since all norms on $X_k$ are mutually equivalent, we find $\delta_k>t_1^{-1}\ (>1)$ such that
	$$
	\delta_k^{-1}|u|_{r(\cdot)}\leq \|u\|_0\leq \delta_k|u|_{r(\cdot)},\quad \forall u\in X_k.
	$$
	For $u\in X_k$ with $\|u\|_0<\delta_k^{-1}$ (thus, $\max\{\|u\|_0,|u|_{r(\cdot)}\}<1$ and $\int_\Omega \frac{1}{p(x)}\big[|\Delta u|^{p(x)} +|\nabla u|^{p(x)}\big] \diff x\leq \|u\|_0^{p^-}<t_1^{p^-}$), we deduce from \eqref{T_lambda.Est2}, \eqref{Est.hat.M0}, $(\mathcal{F}_2)$, \eqref{norm_0-modular} and Proposition~\ref{norm-modular} that
	\begin{align*}
	T_\lambda(u)\leq (M(t_0)+1)\|u\|_0^{p^-}-\frac{\lambda C_2}{q^-}|u|_{r(\cdot)}^{r^+}\leq (M(t_0)+1)\|u\|_0^{p^-}-\frac{\lambda C_2 \delta_k^{-r^+}}{q^-}\|u\|_0^{r^+}.
	\end{align*}
	That is,
	\begin{align*}
	T_\lambda(u)\leq (M(t_0)+1)\|u\|_0^{r^+}\left(\|u\|_0^{p^--r^+}-\frac{\lambda C_2 \delta_k^{-r^+}}{q^-(M(t_0)+1)}\right).
	\end{align*}
	Thus, by taking $\rho$ with $0<\rho<\min\left\{\delta_k^{-1},\left[\frac{\lambda C_2 \delta_k^{-r^+}}{q^-(M(t_0)+1)}\right]^{\frac{1}{p^--r^+}}\right\}$ and 
	$\epsilon:=-(M(t_0)+1)\rho^{p^-}+\frac{\lambda C_2 \delta_k^{-r^+}}{q^-}\rho^{r^+}>0$ we have
	$$
	T_\lambda(u)<-\epsilon < 0,
	$$
	for all $u\in {S_\rho}:=\{u\in X_k: \| u \|_0=\rho \}$. Since
	 $S_\rho\subset T_\lambda^{-\epsilon}$,  we obtain
	$$\gamma( T_\lambda^{-\epsilon})\geq \gamma(S_\rho)=k
	$$
	and this completes the proof.
\end{proof}

Now for each $k \in \mathbb{N}$, define
\begin{gather*}
\Gamma_{k}:=\{A \subset X\backslash\{0\}: A \text{ is closed }, A=-A
\text{ and }  \gamma(A) \geq k\}
\end{gather*}
and 
$$
c_{k}:= \inf_{A\in \Gamma_{k}} \sup_{u \in
	A}T_\lambda(u).
$$
As in \cite{Fig-Nas.EJDE2016}, we derive the following lemma.
\begin{lemma}\label{Le.c_k<0}
	For each $k \in \mathbb{N}$, the number $c_{k}$ is negative.
\end{lemma}
Let $\lambda_\ast^{(2)}>0$ be such that
\begin{equation*}
\left(\frac{1}{p^+}-\frac{1}{q^-}\right)S^{\frac{N}{2}}-K\max\left\{(\lambda_\ast^{(2)})^{\frac{l^+}{l^+-1}},(\lambda_\ast^{(2)})^{\frac{l^-}{l^--1}}\right\}>0,
\end{equation*}
where $K,l$ are given by \eqref{K}. Set 
\begin{equation}\label{lambda_*}
\lambda_\ast:=\min\left\{\lambda_\ast^{(1)},\lambda_\ast^{(2)}\right\},
\end{equation}
where $\lambda_\ast^{(1)}$ is given by \eqref{lambda_*^{(1)}}. The next lemma is derived from Lemma~\ref{Le.PS1} and the deformation lemma, see \cite{Rab1986}.

\begin{lemma}\label{Le.Seq.sol}
	For any $\lambda\in\left(0,\lambda_\ast\right)$, if $c=c_{k}=c_{k+1}=\dots =c_{k+m}$ for some $m \in \mathbb{N}$, then
	$$ \gamma(K_{c})\geq m+1, $$
	where $K_{c}:=\{u \in X\backslash\{0\}: T_\lambda'(u)=0  \text{ and }
	T_\lambda(u)=c\}$.
\end{lemma}

\begin{proof}
	 Let $\lambda\in\left(0,\lambda_\ast\right)$. Then, by Lemma~\ref{Le.c_k<0} and the choice of $\lambda_\ast$ we have
 $$c<0<\left(\frac{1}{p^+}-\frac{1}{q^-}\right)S^{\frac{N}{2}}-K\max\{\lambda^{\frac{l^+}{l^+-1}},\lambda^{\frac{l^-}{l^--1}}\}.$$
  Thus, by Lemmas~\ref{Le.PS1} and \ref{T_lambda(u)<0} it follows that $K_{c}$ is a compact set. The conclusion then follows by a standard argument using the deformation lemma (see, for example, \cite[Lemma 4.4]{Ber-Gar-Per.1996} or \cite[Lemma 3.10]{Fig-Nas.EJDE2016}). 
\end{proof}
\begin{proof}[\textbf{Proof of Theorem \ref{Theo.cc} completed}]
	Let $\lambda_*$ be defined as in \eqref{lambda_*}. Let $\lambda\in (0,\lambda_*)$. By Lemma~\ref{Le.Seq.sol}, $T_\lambda$ admits a sequence $\{u_n\}$ of critical points with $T_\lambda(u_n)<0$ for all $n\in\N$. By Lemma~\ref{T_lambda(u)<0} and \eqref{t_1}, $\{u_n\}$ are also critical points of $\widetilde{J}_\lambda$ with 
	$\int_\Omega \frac{1}{p(x)}|\nabla u_n|^{p(x)} \diff x<t_0$; hence, $\{u_n\}$ are solutions to problem~\eqref{e1.1}. Now, denote by $u_\lambda$ one of $u_n$. By Lemma~\ref{T_lambda(u)<0} again, we have
	$$\int_\Omega \frac{1}{p(x)}\left[|\Delta u_\lambda|^{p(x)} +|\nabla u_\lambda|^{p(x)}\right] \diff x<t_1(\lambda)^{p^-}.$$
	From this and \eqref{lim.t1}, we obtain
	$$\lim_{\lambda\to 0^+}\int_\Omega \frac{1}{p(x)}\left[|\Delta u_\lambda|^{p(x)} +|\nabla u_\lambda|^{p(x)}\right] \diff x=0.$$
	Hence, 
	$$\lim_{\lambda\to 0^+}\|u_\lambda\|_0 =0$$
	in view of \eqref{norm_0-modular}. The desired conclusion then follows due to the equivalence of the norms $\|\cdot\|$ and $\|\cdot\|_0$. The proof is complete.
\end{proof}

\subsection{The $p(\cdot)$-superlinear problem}
In this subsection, we prove the existence of a nontrivial solution to problem~\eqref{e1.1} when the nonlinearity is of $p(\cdot)$-superlinear type. Throughout this section, we always assume that the assumptions $(\mathcal{P})$, $(\mathcal{Q})$, $(\mathcal{C})$, $(\mathcal{M})$, $(\mathcal{F}_0)$, $(\mathcal{F}_3)$ and $(\mathcal{F}_4)$ hold. In this subsection, we define the truncation $M_0(t)$ of $M(t)$ given in \eqref{Def.M0} and the truncated energy functional $\widetilde{J}_\lambda$ given in \eqref{widetilde{J}_lambda} by fixing $t_0>0$ such that 
\begin{equation}\label{Est.Mt0}
m_0 < M(t_0) <\frac{\theta}{p^+}m_0, 
\end{equation}
where $\theta$ is given by assumption $(\mathcal{F}_4)$. We first prove that $\widetilde{J}_\lambda$ has the mountain pass geometry that is stated in the following lemma.

\begin{lemma}\label{lem5.2}
For $\lambda>0$ given, the following assertions hold:
\begin{itemize}
	\item [(i)] there exist two positive constants $\rho$ and $\delta$ such that $\widetilde{J}_\lambda (u) \geq \delta>0$ for all $u\in X$ with $\|u\|_0 =\rho$;
	\item [(ii)]there exists $e_\lambda\in X$ with $\|e_\lambda\|_0 > \rho$ satisfying $\widetilde{J}_\lambda(e_\lambda)<0$.
\end{itemize}
\end{lemma}

\begin{proof}
(i) From $(\mathcal{F}_0)$ and $(\mathcal{F}_3)$, for any $\epsilon>0$, there exists a constant $C(\epsilon)>0$ such that the following estimation
\begin{equation}\label{e5.40}
|F(x,t)| \leq \epsilon|t|^{p(x)} + C(\epsilon) |t|^{\alpha(x)}
\end{equation}
holds true for a.e. $x\in \Omega$ an all $t\in \R$. On the other hand, by Proposition \ref{critical.imb}, we find $C_8>1$ such that 
\begin{equation}\label{norm_0-modular2}
\max\left\{|u|_{\alpha(\cdot)}, |u|_{q(\cdot)}\right\} \leq C_8\|u\|_0, \quad \forall u\in X.
\end{equation}
By relations \eqref{Est.hat.M0} and \eqref{mu1} and \eqref{e5.40}, we have 
\begin{align*}
\notag \widetilde{J}_\lambda(u) & \geq k_0\int_\Omega \frac{1}{p(x)}\left[|\Delta u|^{p(x)}+|\nabla u|^{p(x)}\right]\, \diff x  - \lambda \int_\Omega \left(\epsilon|u|^{p(x)}+C(\epsilon)|u|^{\alpha(x)}\right)\, \diff x \\
\notag & \hspace{1cm}- \int_\Omega \frac{1}{q(x)}|u|^{q(x)}\, \diff x \\
\notag & \geq k_0\int_\Omega \frac{1}{p(x)}\left[|\Delta u|^{p(x)}+|\nabla u|^{p(x)}\right]\, \diff x- \frac{\lambda \epsilon p^+}{\mu_1} \int_\Omega \frac{1}{p(x)}|\nabla u|^{p(x)}\, \diff x\\
&\hspace{1cm} - \lambda C(\epsilon) \int_\Omega |u|^{\alpha(x)}\, \diff x - \frac{1}{q^-}\int_\Omega |u|^{q(x)}\, \diff x,\quad \forall u\in X,
\end{align*}
where $k_0: =\min\{1,m_0\}$. Hence, by choosing $\epsilon = \frac{k_0\mu_1}{2\lambda p^+}>0$, then invoking Proposition~\ref{norm-modular},  \eqref{norm_0-modular} and \eqref{norm_0-modular2}, the preceding estimation yields 
\begin{align*}
\notag\widetilde{J}_\lambda(u)& \geq \frac{k_0}{2}\int_\Omega \frac{1}{p(x)}\left[|\Delta u|^{p(x)}+|\nabla u|^{p(x)}\right]\, \diff x- \lambda C(\epsilon) \int_\Omega |u|^{\alpha(x)}\, \diff x - \frac{1}{q^-}\int_\Omega |u|^{q(x)}\, \diff x\\
\notag & \geq \frac{k_0}{2} \min \left\{\|u\|_0^{p^+},\|u\|_0^{p^-}\right\} - \lambda C(\epsilon)\max\left\{|u|_{\alpha(\cdot)}^{\alpha^+},|u|_{\alpha(\cdot)}^{\alpha^-}\right\} -\frac{1}{q^-}\max\left\{|u|_{q(\cdot)}^{q^+},|u|_{q(\cdot)}^{q^-}\right\} \\
& \geq \frac{k_0}{2}\min \left\{\|u\|_0^{p^+},\|u\|_0^{p^-}\right\} - \lambda C(\epsilon)C_8^{\alpha^+}\max\left\{\|u\|_0^{\alpha^+},\|u\|_0^{\alpha^-}\right\} - \frac{C_8^{q^+}}{q^-}\max\left\{\|u\|_0^{q^+},\|u\|_0^{q^-}\right\}. 
\end{align*}
Hence, for any $u\in X$ with $\|u\|_0 = \rho\in (0,1)$, it holds that
\begin{align}\label{eq5.42}
\notag \widetilde{J}_\lambda(u) \geq &\,\frac{k_0}{2}\rho^{p^+}-\lambda C(\epsilon)C_8^{\alpha^+}\rho^{\alpha^-}-\frac{C_8^{q^+}}{q^-}\rho^{q^-}\\
& = \left(\frac{k_0}{2}-\lambda C(\epsilon)C_8^{\alpha^+}\rho^{\alpha^--p^+}-\frac{C_8^{q^+}}{q^-}\rho^{q^--p^+}\right)\rho^{p^+}=:\delta.
\end{align}
Since $p^+ < \min \left\{\alpha^-, q^-\right\}$, it follows from \eqref{eq5.42} that for $\rho>0$ sufficiently small  we have 
$$\widetilde{J}_\lambda(u) \geq \delta >0,\quad \forall u \in X \ \text{with}\ \|u\|_0 = \rho.$$ 

(ii) Fix $\widetilde{u}_0 \in X$ such that $u>0$ in $\Omega$ and $\|\widetilde{u}_0\|_0 = 1$. Using \eqref{Est.hat.M0} and $(\mathcal{F}_4)$, for any $t>1$, we have
\begin{align}\label{eq5.43}
\notag \widetilde{J}_\lambda(t\widetilde{u}_0) & = \int_\Omega \frac{1}{p(x)}|\Delta (t\widetilde{u}_0)|^{p(x)}\, \diff x+\widehat{M}_0\left(\int_\Omega \frac{1}{p(x)}|\nabla (t\widetilde{u}_0)|^{p(x)}\, \diff x\right) \\
\notag & \hspace{1cm}-\lambda\int_\Omega F(x,t\widetilde{u}_0)\, \diff x-\int_\Omega \frac{1}{q(x)}|t\widetilde{u}_0|^{q(x)}\,\diff x \\
 & \leq t^{p^+}(1+ M(t_0)) - \frac{t^{q^-}}{q^+}\int_\Omega |\widetilde{u}_0|^{q(x)}\,\diff x.
\end{align}
Since $p^+<q^-$, it follows from \eqref{eq5.43} that there exits $t_3(\lambda)>\rho$ such that 
\begin{equation}\label{e5.44}
\widetilde{J}_\lambda(t\widetilde{u}_0)<0,\quad \forall t\geq t_3(\lambda). 
\end{equation}
Then, by setting $e_\lambda := t_3(\lambda)\widetilde{u}_0$ we have that $e_\lambda \in X$ satisfying $\|e_\lambda\|_0 >\rho$ and $\widetilde{J}_\lambda(e_\lambda) < 0$. This completes the proof of Lemma \ref{lem5.2}.
\end{proof}

For each $\lambda>0$, let $e_\lambda$ be as in the preceding lemma and define
\begin{equation}\label{c_lambda}
c_\lambda := \inf_{g \in G_\lambda}\max_{t\in [0,1]}\widetilde{J}_\lambda(g(t)),
\end{equation}
where  
$$G_\lambda := \left\{g \in C([0,1], X):~ g(0) = 0,~ g(1) = e_\lambda\right\}.$$ As a consequence of Lemma \ref{lem5.2} and \cite[Lemma 3.1]{Garcia} we have the following.
\begin{lemma}\label{le.c_lambda}
	The number $c_\lambda$ is positive and there exists a sequence $\{u_n\}\subset X$ such that 
	$$
	\widetilde{J}_\lambda(u_n) \to c_{\lambda}~ \text{ and } ~ \widetilde{J}'_\lambda(u_n) \to 0~ \text{ as } ~ n\to\infty.
	$$ 
\end{lemma}
Furthermore, we have the following property for $c_\lambda$. This result, together with Theorem~\ref{Theo.ccp}, will help us to overcome the lack of compactness in the $p(\cdot)$-superlinear case of $f$.
\begin{lemma}\label{lem5.3}
It holds that
$$
\lim_{\lambda \to \infty}c_\lambda = 0,
$$
where $c_\lambda$ is given by \eqref{c_lambda}.
\end{lemma}

\begin{proof}
Let $\{\lambda_n\}$ be an abitrary sequence of real positive numbers such that $\lambda_n \to \infty$ as $n\to \infty$. By the proof of Lemma \ref{lem5.2}, for each $n \in \mathbb{N}$, there exists $t_{\lambda_n} > 0$ such that $\widetilde{J}_{\lambda_n} (t_{\lambda_n}\widetilde{u}_0) = \max_{t \geq 0}\widetilde{J}_{\lambda_n}(t\widetilde{u}_0)$. For this reason, it follows that $t_{\lambda_n}\frac{d}{dt} \widetilde{J}_{\lambda_n} (t\widetilde{u}_0) \Big|_{t=t_{\lambda_n}} = \langle \widetilde{J}_{\lambda_n}' (t_{\lambda_n}\widetilde{u}_0), t_{\lambda_n}\widetilde{u}_0\rangle = 0$, namely, 
\begin{align}\label{e5.46}
\notag & \int_\Omega |\Delta (t_{\lambda_n}\widetilde{u}_0)|^{p(x)}\, \diff x + M_0\left(\int_\Omega \frac{1}{p(x)}|\nabla (t_{\lambda_n}\widetilde{u}_0)|^{p(x)}\, \diff x\right)\int_\Omega |\nabla (t_{\lambda_n}\widetilde{u}_0)|^{p(x)}\, \diff x \\
 & \hspace{1cm} = {\lambda_n} \int_\Omega f(x,t_{\lambda_n}\widetilde{u}_0)t_{\lambda_n} \widetilde{u}_0\, \diff x + \int_\Omega |t_{\lambda_n} \widetilde{u}_0|^{q(x)}\, \diff x.
\end{align}
Thus, by $(\mathcal{F}_4)$ we obtain
\begin{align}\label{e5.47}
\notag \int_\Omega |\Delta (t_{\lambda_n}\widetilde{u}_0)|^{p(x)}\, \diff x & + M_0\left(\int_\Omega \frac{1}{p(x)}|\nabla (t_{\lambda_n}\widetilde{u}_0)|^{p(x)}\, \diff x\right)\int_\Omega |\nabla (t_{\lambda_n}\widetilde{u}_0)|^{p(x)}\, \diff x\\
&\geq \min\left\{t_{\lambda_n}^{q^+},t_{\lambda_n}^{q^-}\right\}\int_\Omega |\widetilde{u}_0|^{q(x)}\, \diff x.
\end{align}
On the other hand, taking into account \eqref{Est.hat.M0}, \eqref{norm_0-modular} and the fact that $\|\widetilde{u}_0\|_0 = 1$, we get
\begin{align}\label{e5.48}
\notag & \int_\Omega |\Delta (t_{\lambda_n}\widetilde{u}_0)|^{p(x)}\, \diff x + M_0\left(\int_\Omega \frac{1}{p(x)}|\nabla (t_{\lambda_n}\widetilde{u}_0)|^{p(x)}\, \diff x\right)\int_\Omega |\nabla (t_{\lambda_n}\widetilde{u}_0)|^{p(x)}\, \diff x \\
\notag & \hspace{1cm} \leq p^+\max\left\{1,M(t_0)\right\} \int_\Omega \frac{1}{p(x)}\left[|\Delta (t_{\lambda_n} \widetilde{u}_0)|^{p(x)}+|\nabla (t_{\lambda_n}\widetilde{u}_0)|^{p(x)}\right]\, \diff x\\
& \hspace{1cm} \leq p^+\max\left\{1,M(t_0)\right\} \max\left\{t_{\lambda_n}^{p^+},t_{\lambda_n}^{p^-}\right\}.
\end{align}
Using \eqref{e5.47} and \eqref{e5.48}, we deduce that the sequence $\{t_{\lambda_n}\}$ is bounded since $p^+ < q^-$. Up to a subsequence, we may assume that $t_{\lambda_n} \to \widetilde{t}_0$ as $n\to \infty$. Moreover, by \eqref{e5.46} and \eqref{e5.48} we have
\begin{equation}\label{eq5.49}
{\lambda_n} \int_\Omega f(x,t_{\lambda_n}\widetilde{u}_0)t_{\lambda_n} \widetilde{u}_0\, \diff x + \int_\Omega |t_{\lambda_n} \widetilde{u}_0|^{q(x)}\, \diff x \leq C, \quad \forall n\in \mathbb{N}.
\end{equation}
If $\widetilde{t}_0 >0$, then it follows from assumption $(\mathcal{F}_4)$ that 
$$
 {\lambda_n} \int_\Omega f(x,t_{\lambda_n}\widetilde{u}_0)t_{\lambda_n} \widetilde{u}_0\, \diff x + \int_\Omega |t_{\lambda_n} \widetilde{u}_0|^{q(x)}\, \diff x \to \infty, \quad n\to \infty,
$$
which contradicts \eqref{eq5.49}. So, we get $\widetilde{t}_0 = 0$, namely,
\begin{equation}\label{lim.t_lambda_n}
\lim_{n\to\infty}\, t_{\lambda_{n}}=0.
\end{equation}

For each $n\in\N$, we consider the path $g_\ast(t) = te_{\lambda_n}$ with  $t \in [0,1]$, where $e_{\lambda_n}$ is taken from the proof of Lemma~\ref{lem5.2}. Clearly, $g_\ast \in G_{\lambda_n}$ and note that, by applying \eqref{e5.44} for $\lambda=\lambda_n$,
$$\max_{t\geq 0} \widetilde{J}_{\lambda_n}(t\widetilde{u}_0)=\max_{t\in [0,t_3(\lambda_n)]} \widetilde{J}_{\lambda_n}(t\widetilde{u}_0)= \max_{t\in [0,1]} \widetilde{J}_{\lambda_n}(te_{\lambda_n})=\max_{t\in[0,1]}\widetilde{J}_{\lambda_n}(g_\ast(t)).$$
Thus, by $(\mathcal{F}_4)$, \eqref{Est.hat.M0} and \eqref{e5.48}, we have the following estimate
\begin{align}\label{e5.51}
\notag 0 < c_{\lambda_n} & = \inf_{g \in G_{\lambda_n}}\max_{t\in [0,1]}\widetilde{J}_{\lambda_n}(g(t)) \\
\notag & \leq \max_{t\in [0,1]} \widetilde{J}_{\lambda_n}(g_\ast(t))=\max_{t\geq 0} \widetilde{J}_{\lambda_n}(t\widetilde{u}_0)= \widetilde{J}_{\lambda_n}(t_{\lambda_n}\widetilde{u}_0) \\
\notag&\leq \max\left\{1,M(t_0)\right\} \int_\Omega \frac{1}{p(x)}\left[|\Delta (t_{\lambda_n} \widetilde{u}_0)|^{p(x)}+|\nabla (t_{\lambda_n}\widetilde{u}_0)|^{p(x)}\right]\, \diff x\\
& \leq \max\left\{1,M(t_0)\right\} \max\left\{t_{\lambda_n}^{p^+},t_{\lambda_n}^{p^-}\right\}.
\end{align}
From \eqref{lim.t_lambda_n} and \eqref{e5.51} we obtain  $\lim_{n \to \infty}c_{\lambda_n} = 0$ and thus, the desired conclusion follows.
\end{proof}

\begin{proof}[\textbf{Proof of Theorem \ref{Theo.sl}}]
By Lemma \ref{lem5.3}, there exists $\widetilde{\lambda}^*>0$ such that 
\begin{equation}\label{e5.52}
c_\lambda < \left(\frac{1}{\theta}-\frac{1}{q^-}\right)S^\frac{N}{2},\quad \forall \lambda \geq \widetilde{\lambda}^*,
\end{equation}
where $c_\lambda$ and $S$ are given by \eqref{c_lambda} and \eqref{S}, respectively. We will show that for each $\lambda \geq \widetilde{\lambda}^*$,   
the modified energy functional $\widetilde{J}_\lambda$ has a nontrivial critical point. To this end, let $\lambda \geq \widetilde{\lambda}^*$. By Lemma~\ref{le.c_lambda}, there exists a sequence $\{u_n\}\subset X$ such that
\begin{equation}\label{e5.53}
\widetilde{J}_\lambda(u_n) \to c_{\lambda}\ (>0)~ \text{ and } ~ \widetilde{J}'_\lambda(u_n) \to 0~ \text{ as } ~ n\to\infty.
\end{equation}
We claim that $\{u_n\}$ is bounded in $X$. 
 Indeed, using \eqref{e5.53}, \eqref{Est.M0}, \eqref{Est.hat.M0} and assumption $(\mathcal{F}_4)$, we have that for all $n\in \mathbb{N}$ large enough,
\begin{align*}
\notag & 1+ c_\lambda + \|u_n\|_0 \\
\notag & \geq \widetilde{J}_\lambda(u_n) - \frac{1}{\theta}\langle \widetilde{J}_\lambda'(u_n), u_n\rangle \\
\notag & \geq \left(\frac{1}{p^+}-\frac{1}{\theta}\right) \int_\Omega |\Delta u_n|^{p(x)}\,\diff x + \widehat{M}_0\left(\int_\Omega \frac{1}{p(x)}|\nabla u_n|^{p(x)}\, \diff x\right) \\
\notag & \hspace{1cm}-\frac{1}{\theta} M_0\bigg(\int_\Omega\frac{1}{p(x)}|\nabla u_n|^{p(x)}\,\diff x\bigg)\int_{\Omega}|\nabla u_n|^{p(x)}\diff x + \left(\frac{1}{\theta} - \frac{1}{q^-}\right)\int_{\Omega}|u_n|^{q(x)}\diff x \\
\notag & \hspace{1cm}+\lambda \int_\Omega \left(\frac{1}{\theta}f(x,u_n)u_n - F(x,u_n)\right)\, \diff x \\
\notag & \geq \left(\frac{1}{p^+}-\frac{1}{\theta}\right) \int_\Omega |\Delta u_n|^{p(x)}\,\diff x + \left(\frac{m_0}{p^+}-\frac{M(t_0)}{\theta}\right)\int_\Omega |\nabla u_n|^{p(x)}\diff x\\
& \geq \widetilde{k}_0 \int_\Omega \frac{1}{p(x)}\left[|\Delta u_n|^{p(x)}+|\nabla u_n|^{p(x)}\right]\,\diff x,
\end{align*}
where $\widetilde{k}_0 :=p^- \min\left\{\left(\frac{1}{p^+}-\frac{1}{\theta}\right), \left(\frac{m_0}{p^+} -\frac{M(t_0)}{\theta} \right) \right\} >0$ (see \eqref{Est.Mt0}). From this and \eqref{norm_0-modular} we deduce that for all $n\in \mathbb{N}$ large enough,
\begin{equation*}
1+ c_\lambda + \|u_n\|_0 \geq  \widetilde{k}_0 \left(\|u_n\|^{p^-}_0-1\right),
\end{equation*}
which yields the boundedness of $\{u_n\}$ in $X$ due to $p^->1$. By the reflexivity of $X$ and Theorem~\ref{Theo.ccp}, up to a subsequence we have
\begin{eqnarray}
u_n(x) &\to& u^\lambda(x)  \quad \text{a.e.} \ \ x\in\Omega,\label{PL.PS2.a.e}\\
u_n &\rightharpoonup& u^\lambda  \quad \text{in} \  X,\label{PL.PS2.w-conv}\\
|\Delta u_n|^{p(\cdot)} &\overset{\ast }{\rightharpoonup }&\mu \geq |\Delta u^\lambda|^{p(\cdot)} + \sum_{i\in I} \mu_i \delta_{x_i} \ \text{in}\  \mathcal{M}(\overline{\Omega}),\label{PL.PS2.mu}\\
|u_n|^{q(\cdot)}&\overset{\ast }{\rightharpoonup }&\nu=|u^\lambda|^{q(\cdot)} + \sum_{i\in I}\nu_i\delta_{x_i} \ \text{in}\ \mathcal{M}(\overline{\Omega}),\label{PL.PS2.nu}\\
S \nu_i^{\frac{1}{p_2^\ast(x_i)}} &\leq& \mu_i^{\frac{1}{p(x_i)}}, \ \forall i\in I.\label{PL.PS2.mu-nu}
\end{eqnarray}
As before, we will show that $I = \emptyset$. Indeed, suppose on the contrary that there exists $i\in I$. By repeating arguments used in the proof Lemma~\ref{Le.PS1} with $u=u^\lambda$ we obtain 
\begin{equation}\label{est.mu_nu2}
\nu_i=\mu_i \geq S^\frac{N}{2}.
\end{equation}
On the other hand we deduce from \eqref{e5.53}, $(\mathcal{F}_4)$, \eqref{Est.M0} and \eqref{Est.hat.M0} that
\begin{align}\label{e5.60}
\notag c_\lambda + o_n(1) & = \widetilde{J}_\lambda(u_n) - \frac{1}{\theta}\langle \widetilde{J}_\lambda'(u_n), u_n \rangle \\
\notag & \geq \left(\frac{1}{p^+}-\frac{1}{\theta}\right) \int_\Omega |\Delta u_n|^{p(x)}\,\diff x + \left(\frac{m_0}{p^+}-\frac{M(t_0)}{\theta}\right)\int_\Omega |\nabla u_n|^{p(x)}\diff x \\
 & \hspace{1cm} + \left(\frac{1}{\theta} - \frac{1}{q^-}\right)\int_{\Omega}|u_n|^{q(x)}\diff x. 
\end{align}
Hence,
\begin{align*}
 c_\lambda + o_n(1) \geq \left(\frac{1}{\theta} - \frac{1}{q^-}\right)\int_{\Omega}|u_n|^{q(x)}\diff x. 
\end{align*}
Now, passing to the limit as $n\to \infty$ in the last inequality and utilizing  \eqref{PL.PS2.nu} and \eqref{est.mu_nu2} we obtain
$$
c_\lambda \geq \left(\frac{1}{\theta} - \frac{1}{q^-}\right)\nu_i\geq \left(\frac{1}{\theta} - \frac{1}{q^-}\right)S^\frac{N}{2},
$$ 
which is in a contradiction with \eqref{e5.52}. Thus, we have proved that $I = \emptyset$. Again,  arguing as in the proof Lemma~\ref{Le.PS1} we infer that $u_n \to u^\lambda$ in $X$ as $n\to \infty$. Hence, by \eqref{e5.53} we get that $u^\lambda$ is a critical point of the modified energy functional $\widetilde{J}_\lambda$. Moreover, $u^\lambda$ is nontrivial since $\widetilde{J}_\lambda(u^\lambda)=c_\lambda>0$.

In order to finish the proof of Theorem \ref{Theo.sl}, we will show that there exists $\lambda^\ast \geq \widetilde{\lambda}^*$ such that for each $\lambda \geq \lambda^\ast$, the above $u^\lambda$ is also a nontrivial solution to problem \eqref{e1.1}. Indeed, by \eqref{e5.60} again we have 
\begin{align*}
c_{\lambda} + o_n(1) &\geq \left(\frac{1}{p^+}-\frac{1}{\theta}\right) \int_\Omega |\Delta u_n|^{p(x)}\,\diff x + \left(\frac{m_0}{p^+}-\frac{M(t_0)}{\theta}\right)\int_\Omega |\nabla u_n|^{p(x)}\diff x\\
&\geq \widetilde{k}_0 \int_\Omega \frac{1}{p(x)}\left[|\Delta u_n|^{p(x)}+ |\nabla u_n|^{p(x)}\right]\diff x,
\end{align*}
where $\widetilde{k}_0 =p^- \min\left\{\left(\frac{1}{p^+}-\frac{1}{\theta}\right), \left(\frac{m_0}{p^+} -\frac{M(t_0)}{\theta} \right) \right\} >0$. Passing to the limit as $n\to \infty$ in the last inequality we arrive at
\begin{equation}
\notag c_{\lambda} \geq  \widetilde{k}_0 \int_\Omega \frac{1}{p(x)}\left[|\Delta u^\lambda|^{p(x)}+ |\nabla u^\lambda|^{p(x)}\right]\diff x
\end{equation}
and thus,
\begin{equation}\label{e5.61}
 \int_\Omega \frac{1}{p(x)}\left[|\Delta u^\lambda|^{p(x)}+ |\nabla u^\lambda|^{p(x)}\right]\diff x\leq   \widetilde{k}_0^{-1}c_{\lambda}
\end{equation}
From this and Lemma~\ref{lem5.3}, we find $\lambda^\ast \geq \widetilde{\lambda}^*$ such that 
\begin{equation}
\notag \int_\Omega \frac{1}{p(x)}\left[|\Delta u^\lambda|^{p(x)}+ |\nabla u^\lambda|^{p(x)}\right]\diff x\leq t_0,\quad \forall   \lambda \geq \lambda^\ast.
\end{equation}
This means that $u^\lambda$ is also a nontrivial solution to problem \eqref{e1.1} provided $\lambda \geq \lambda^\ast$. Moreover, \eqref{e5.61} also implies that $\lim_{\lambda \to \infty}\|u^\lambda\|_0 =0$ in view of Lemma~\ref{lem5.3} and the relation \eqref{norm_0-modular}. By the equivalence of two norms $\|\cdot\|_0$ and $\|\cdot\|$ in $X$, we complete the proof of Theorem \ref{Theo.sl}.
\end{proof}

\vspace{0.5cm}
\noindent{\bf Acknowledgment}\\

The second author was supported by University of Economics Ho Chi Minh City, Vietnam.

\end{document}